\documentclass[11pt]{amsart}
\usepackage{amsmath}
\usepackage{amssymb}
\usepackage{amsthm}
\usepackage{amscd}
\usepackage{enumerate}
\usepackage{enumitem}
\usepackage{verbatim}
\usepackage[all]{xy}
\usepackage{mathrsfs}
\usepackage{mathabx}
\usepackage{url}

\theoremstyle{plain}
\newtheorem{thm}{Theorem}[section]
\newtheorem{prop}[thm]{Proposition}
\newtheorem{lem}[thm]{Lemma}
\newtheorem{cor}[thm]{Corollary}

\theoremstyle{definition}
\newtheorem{dfn}[thm]{Definition}
\newtheorem{rmk}[thm]{Remark}

\newcommand{\AD}{\mathrm{AD}}

\newcommand{\alg}{\mathrm{alg}}

\newcommand{\Aut}{\mathrm{Aut}}

\newcommand{\Cot}{\mathrm{Cot}}

\newcommand{\Cusps}{\mathrm{Cusps}}

\newcommand{\Der}{\mathrm{Der}}

\newcommand{\dR}{\mathrm{dR}}
\newcommand{\DR}{\mathrm{DR}}

\newcommand{\End}{\mathrm{End}}

\newcommand{\Hom}{\mathrm{Hom}}

\newcommand{\id}{\mathrm{id}}
\newcommand{\inner}{\mathrm{in}}

\newcommand{\Img}{\mathrm{Im}}

\newcommand{\Isom}{\mathrm{Isom}}
\newcommand{\KS}{\mathrm{KS}}
\newcommand{\Ker}{\mathrm{Ker}}

\newcommand{\Lie}{\mathrm{Lie}}

\newcommand{\prjt}{\mathrm{pr}}

\newcommand{\Res}{\mathrm{Res}}

\newcommand{\si}{\mathrm{si}}

\newcommand{\Spec}{\mathrm{Spec}}

\newcommand{\Sym}{\mathrm{Sym}}

\newcommand{\TD}{\mathrm{TD}}

\newcommand{\td}{\triangledown}

\newcommand{\univ}{\mathrm{un}}

\newcommand{\sHom}{\mathscr{H}\!\!\mathit{om}}
\newcommand{\sIsom}{\mathscr{I}\!\!\mathit{som}}

\newcommand{\cC}{\mathcal{C}}
\newcommand{\cD}{\mathcal{D}}
\newcommand{\cE}{\mathcal{E}}
\newcommand{\cF}{\mathcal{F}}
\newcommand{\cG}{\mathcal{G}}
\newcommand{\cH}{\mathcal{H}}

\newcommand{\cL}{\mathcal{L}}
\newcommand{\cM}{\mathcal{M}}

\newcommand{\cO}{\mathcal{O}}
\newcommand{\cP}{\mathcal{P}}

\newcommand{\cR}{\mathcal{R}}

\newcommand{\cZ}{\mathcal{Z}}

\newcommand{\Ga}{\mathbb{G}_\mathrm{a}}

\newcommand{\vep}{\varepsilon}

\newcommand{\bA}{\mathbb{A}}
\newcommand{\bB}{\mathbb{B}}

\newcommand{\bF}{\mathbb{F}}

\newcommand{\bP}{\mathbb{P}}

\newcommand{\bV}{\mathbb{V}}

\newcommand{\sB}{\mathscr{B}}

\newcommand{\sM}{\mathscr{M}}

\newcommand{\sT}{\mathscr{T}}

\newcommand{\frn}{\mathfrak{n}}

\newcommand{\Kinf}{K_\infty}
\newcommand{\Cinf}{\mathbb{C}_{\infty}}

\makeatletter
    
    \@addtoreset{equation}{section}
\makeatother

\makeatletter
\renewcommand{\p@enumii}{}
\makeatother

\begin{document}

\title[On autoduality of Drinfeld modules]{On autoduality of Drinfeld modules and Drinfeld modular forms}

\author{Shin Hattori}
\address{Department of Natural Sciences, Tokyo City University, 1-28-1 Tamazutsumi, Setagaya-ku, Tokyo 158-8557, Japan}
\email{hattoris@tcu.ac.jp}


\date{\today}


\begin{abstract}
	Let $\bF_q$ be the field of $q$ elements and let $A=\bF_q[t]$ be the polynomial ring over $\bF_q$. Let $\frn\in A\setminus \bF_q$ be a monic polynomial with a 
	prime factor of degree prime to $q-1$. Let
	$\Delta$ be a subgroup of $(A/(\frn))^\times$ such that the map $\Delta\to (A/(\frn))^\times/\bF_q^\times$ is bijective. Let $S$ be a scheme over $A[1/\frn]$ 
	and let $R$ be an $A[1/\frn]$-algebra which is an excellent regular domain.
	In this paper, we show that any Drinfeld module $E$ of rank two over $S$ admitting a $\Gamma_1^\Delta(\frn)$-structure is isomorphic to its Taguchi dual $E^D$. 
	As an application, 
	for the Hodge bundle $\bar{\omega}$ on the Drinfeld modular curve $X$ of level $\Gamma_1^\Delta(\frn)$ over $R$, we give a dual 
	Kodaira--Spencer isomorphism of the form $\bar{\omega}^{\otimes 2}\simeq \Omega^1_{X/R}(2\Cusps)$, in contrast with the usual one 
	in the Drinfeld case in which $E^D$ is involved.
\end{abstract}

\maketitle



\section{Introduction}\label{SecIntro}

Let $p$ be a rational prime, $q>1$ a power of $p$ and $\bF_q$ the field of $q$ elements. Let $t$ be an indeterminate, $A=\bF_q[t]$, $K=\bF_q(t)$, $K_\infty=\bF_q((1/t))$ and 
$\Cinf$ the $(1/t)$-adic completion of an algebraic closure of $\Kinf$. Let $\Omega=\Cinf\setminus\Kinf$ be the Drinfeld upper half plane,
equipped with a natural structure of a rigid analytic variety over $\Cinf$.

Drinfeld modular forms are analogues over $\bF_q(t)$ of elliptic modular forms. They are rigid analytic functions on $\Omega$ satisfying a transformation 
condition and regularity at cusps similar to 
those for elliptic modular forms. 
In the theory of Drinfeld modular forms, Drinfeld modules of rank two, which are line bundles with an exotic action of $A$, 
play the role that elliptic curves play in the theory of elliptic modular forms. 
Though harmonic cocycles provide a powerful tool to understand Drinfeld modular forms combinatorially \cite{Tei}, a geometric 
approach using Drinfeld modular curves and their compactifications still seems meaningful to investigate (see for example \cite{BSCR,Boeckle,dVr,GV, Ha_HTT}).

One problem we encounter in the geometric study of Drinfeld modular forms is the lack of autoduality for Drinfeld modules. 
Any elliptic curve $\cE$ has autoduality: there exists a natural isomorphism from $\cE$ to its dual.   
On the other hand, there is a notion of a dual Drinfeld module due to Taguchi \cite{Tag}. For a Drinfeld module $E$ of rank two, the dual $E^D$ is again a Drinfeld 
module of rank two.
However, we do not have an isomorphism $E\simeq E^D$ of Drinfeld modules in general. 

This lack affects the theory in different ways, especially on the shape of the Kodaira--Spencer isomorphism in the Drinfeld case. 
For any Drinfeld module $E$, we denote by $\omega_E$ the sheaf of invariant differentials on $E$ and by $\cH_{\dR}(E)$ 
Gekeler's de Rham sheaf of $E$ 
\cite[Definition 3.4]{Gek_dR}, so that we have a 
Hodge filtration 
\[
\xymatrix{
0\ar[r] & \omega_E \ar[r] & \cH_\dR(E)\ar[r] & \omega^\vee_{E^D}\ar[r] & 0.
}
\]

For any $\frn\in A\setminus \bF_q$, 
we have a Drinfeld modular curve $Y(\frn)
$ over $A[1/\frn]$ of full level $\frn$ classifying pairs of a Drinfeld module of rank two and a $\Gamma(\frn)$-structure on it. Let $E_\univ$ be the universal 
Drinfeld 
module over $Y(\frn)$. 
Then, using the Hodge filtration on $\cH_\dR(E_\univ)$, Gekeler \cite[Theorem 6.11]{Gek_dR} defined the (dual) Kodaira--Spencer isomorphism as
\[
\KS^\vee: \omega_{E_\univ}\otimes_{\cO_{Y(\frn)}} \omega_{E^D_\univ}\to\Omega^1_{Y(\frn)/A[1/\frn]}, 
\]
in contrast with the case of elliptic modular curves where we have such an isomorphism of the form $\omega_{E_\univ}^{\otimes 2}\simeq \Omega^1_{Y(\frn)/A[1/\frn]}
$. 

Since Drinfeld modular forms are sections of tensor powers of $\omega_{E_\univ}$, it is desirable to have a Kodaira--Spencer isomorphism of the latter
form. The aim of this paper is to construct such an isomorphism.

Let $R$ be an $A[1/\frn]$-algebra which is an excellent regular domain. Suppose that $\frn$ has a prime factor of degree prime to $q-1$.
Let $\Delta\subseteq (A/(\frn))^\times$ be a subgroup such that the map $\Delta\to  (A/(\frn))^\times/\bF_q^\times$ is an isomorphism.
In \cite{Ha_DMC}, a level structure over $A[1/\frn]$ called a $\Gamma_1^\Delta(\frn)$-structure is studied.
It is a pair $(\lambda,\mu)$ of a $\Gamma_1(\frn)$-structure $\lambda$ and an additional datum $\mu$ such that the definition of $\mu$ depends
on the choice of $\Delta$
(Definition \ref{DfnLevelGamma1NDeltaStr}).

Let $Y_1^\Delta(\frn)_R$ be the Drinfeld modular curve of level $\Gamma_1^\Delta(\frn)$ over $R$ and
let $X_1^\Delta(\frn)_R$ be its compactification. Let $E_\univ^\Delta$ be the universal Drinfeld module of rank two over $Y_1^\Delta(\frn)_R$.
Let $\omega_\univ^\Delta=\omega_{E^\Delta_\univ}$ and $\cH_{\dR,\univ}^\Delta=\cH_\dR(E_\univ^\Delta)$. We denote by $\bar{\omega}_\univ^\Delta$ 
the natural extension of $\omega_\univ^\Delta$ to $X_1^\Delta(\frn)_R$.
Let $\Cusps_R^\Delta$ be the effective Cartier divisor of cusps on $X_1^\Delta(\frn)_R$.
Then the main theorem of this paper is the following.

\begin{thm}\label{ThmMain}
	\begin{enumerate}
		\item\label{ThmMain-Autodual} (Theorem \ref{ThmAutodual}) Let $S$ be a scheme over $A[1/\frn]$. Then any Drinfeld module $E$ of rank two over $S$ admitting 
		a $\Gamma_1^\Delta(\frn)$-structure has 
		an 
		autoduality. Namely, for any $\Gamma_1^\Delta(\frn)$-structure $(\lambda,\mu)$ on $E$, there
		exists an isomorphism 
		\[
		\AD_{(E,\lambda,\mu)}: E\to E^D
		\] 
		of Drinfeld modules over $S$.
		
			\item\label{ThmMain-Hodge} (Theorem \ref{ThmExtProjToLie}) 
			Let $\AD_{E_\univ^\Delta}: E_\univ^\Delta\to (E_\univ^\Delta)^D$ be the isomorphism of (\ref{ThmMain-Autodual}) for the universal object
			$(E_\univ^\Delta,\lambda_\univ,\mu_\univ)$ over $Y_1^\Delta(\frn)_R$. Then
			the exact sequence 
		\[
		\xymatrix{
			0 \ar[r] & \omega_\univ^\Delta \ar[r] & \cH_{\dR,\univ}^\Delta \ar[r]^{\prjt} & (\omega_{\univ}^\Delta)^\vee \ar[r] &0
		}
		\]
		obtained from the Hodge filtration by composing the inverse of 
		the dual $(\AD_{E^\Delta_\univ}^*)^\vee$ of the isomorphism $
		\AD_{E^\Delta_\univ}^*:\omega_{(E^\Delta_\univ)^D}\to \omega_{E_\univ^\Delta}$
		extends to an exact sequence
		\[
		\xymatrix{
			0 \ar[r] & \bar{\omega}_\univ^\Delta \ar[r] & \bar{\cH}_{\dR,\univ}^\Delta \ar[r]& (\bar{\omega}_{\univ}^\Delta)^\vee \ar[r] &0
		}
		\]
		of finite locally free sheaves on $X_1^\Delta(\frn)_R$.
		
		\item (Theorem \ref{ThmKS}) Let \[
		\KS^\vee: \omega_{E_\univ^\Delta}\otimes_{\cO_{Y_1^\Delta(\frn)_R}} \omega_{(E^\Delta_\univ)^D}\to\Omega^1_{Y_1^\Delta(\frn)_R/R}
		\]
		be the Kodaira--Spencer isomorphism for $E^\Delta_\univ$. 
		Then the isomorphism $\KS^\curlyvee: (\omega_\univ^\Delta)^{\otimes 2}\to \Omega^1_{Y_1^\Delta(\frn)_R/R}$ obtained by composing $
		\KS^\vee$ and $
		(\AD_{E^\Delta_\univ}^*)^{-1}$ 
		extends to an isomorphism
		\[
		\bar{\KS}^\curlyvee:  (\bar{\omega}_\univ^\Delta)^{\otimes 2}\to \Omega^1_{X_1^\Delta(\frn)_R/R}(2\Cusps_R^\Delta)
		\]
		of invertible sheaves on $X_1^\Delta(\frn)_R$.

			\item (Theorem \ref{ThmExtDeRhamPair}) There exists an $\cO_{X_1^\Delta(\frn)_R}$-linear
			alternating perfect pairing
			\[
			\overline{\langle-,-\rangle}_\univ^\Delta:\bar{\cH}_{\dR,\univ}^\Delta\otimes_{\cO_{X_1^\Delta(\frn)_R}}\bar{\cH}_{\dR,\univ}^\Delta\to 
			\cO_{X_1^\Delta(\frn)_R}
			\]
			which induces the canonical pairing
			\[
			\bar{\omega}^\Delta_\univ\otimes_{\cO_{X_1^\Delta(\frn)_R}} (\bar{\omega}^\Delta_\univ)^\vee\to\cO_{X_1^\Delta(\frn)_R}
			\]
			via the exact sequence of (\ref{ThmMain-Hodge}).
	\end{enumerate}
\end{thm}

Let us briefly explain an idea of the proof of Theorem \ref{ThmMain}.
	Let $S$ be a scheme over $A[1/\frn]$ and let $E$ be a Drinfeld module of rank two over $S$. Let $\cL_E$ be the underlying invertible sheaf of $E$. Then the 
underlying 
invertible sheaf of $E^D$ is $\cL_E^{\otimes -q}$.
Write the morphism of multiplying $t$ on $E$ as
\[
\Phi^E_t=\theta+ \alpha_1^E \tau+ \alpha_2^E\tau^2,\quad \alpha_i^E\in \cL_E^{\otimes 1-q^i}(S).
\]
It is known \cite[Remark 2.20]{Ha_HTT} that if we have a section $H\in \cL_E^{\otimes -1-q}(S)$ satisfying $H^{\otimes q-1}=-\alpha^E_2$,
then multiplying $H$ gives an autoduality.

A key observation is that, over $\Cinf$, Gekeler's $h$-function gives such a section, since the $h$-function 
satisfies the relation $h^{q-1}=-g_2$ for the discriminant function 
$g_2$. Then the $x$-expansion principle \cite[Proposition 4.8 (i)]{Ha_HTT} assures that the section descends to that on $X_1^\Delta(\frn)_R$, which gives
the autoduality of Theorem \ref{ThmMain} (\ref{ThmMain-Autodual}). The other assertions follow from gluing in the style of Beauville--Laszlo \cite{BeauLasz} and
a study around cusps using Tate--Drinfeld modules.

The organization of the paper is as follows. In \S\ref{SecDM}, we recall the definitions of Drinfeld modules, Drinfeld modular curves involved and Tate--Drinfeld 
modules. In \S\ref{SecKS}, we recall Gekeler's theory of biderivations, de Rham sheaves and Kodaira--Spencer maps for Drinfeld modules. We also give an explicit 
description of the Hodge filtration of a Drinfeld module over an affine scheme when the underlying invertible sheaf is trivial. In \S\ref{SecExt}, we prove a 
lemma (Lemma \ref{LemExt}) to construct 
extensions in Theorem \ref{ThmMain}, and also give a natural extension of $\cH_{\dR,\univ}^\Delta$ to $X_1^\Delta(\frn)_R$. 
Then the main theorems are proved in \S\ref{SecAD}. In the course of the proof, we also fill some gaps in \cite{Ha_HTT,Ha_DMC} concerning the base change 
compatibility
of sheaves of differentials by using the $F$-finiteness of $A[1/\frn]$ (Remarks \ref{RmkFixHTT} and \ref{RmkFixDMC}).
We also present an intrinsic construction of $Y_1^\Delta(\frn)_R$ independently of the choice of $\Delta$ (Theorem \ref{ThmHstrVSGammaDeltaStr}),
though it is not used in the proof of the main theorems.

\subsection*{Acknowledgements} The author would like to thank O\u{g}uz Gezm\.{i}\c{s} and Sriram Chinthalagiri Venkata for asking about the extension of 
the de Rham sheaf 
and related discussions, which encouraged him to write this paper. This work was supported by JSPS KAKENHI Grant Number JP23K03078. 

\section{Drinfeld modules and Drinfeld modular curves}\label{SecDM}

\subsection{Drinfeld modules}\label{SubsecDM}

Let $S$ be a scheme over $\bF_q$. For an $\cO_S$-module $\cF$ and schemes $X$ and $T$ over $S$, we denote by $\cF|_T$ and $X|_T$ the pull-backs of $\cF$ and $X$ by 
the morphism
$T\to S$. When $T=\Spec(B)$, we often denote them by $\cF|_B$ and $X|_B$.
We write $\cF^{(q)}$ and $X^{(q)}$ for the pull-backs of $\cF$ and $X$ by the absolute $q$-th power Frobenius morphism $\sigma_S:S\to S$.
For any $\bF_q$-algebra $B$, the absolute $q$-th power Frobenius map on $B$ is denoted by $\sigma$. For any scheme $S$ over $A$, we denote by $\theta$ the image of 
$t$ by the structure map $A\to \cO_S(S)$. For any $\cO_S$-module $\cF$ and any $B$-module $M$, we denote by $\cF^\vee=\sHom_{\cO_S}(\cF,\cO_S)$ and 
$M^\vee=\Hom_B(M,B)$ their linear duals. When $S=\Spec(B)$ is affine, we often identify a quasi-coherent sheaf $\cF$ on $S$ with the $B$-module $\cF(S)$ abusively.

For a group scheme $\cG$ over a scheme $S$, we denote by 
\[
\Cot(\cG)=\omega_{\cG},\quad \Lie(\cG)
\]
the sheaf of invariant differentials on $\cG$ and the Lie algebra of $\cG$, respectively. 
We have $\Lie(\cG)=\omega_{\cG}^\vee$, and if the $\cO_S$-module $\omega_\cG$ is 
finite locally 
free, then the natural homomorphism $\omega_\cG\to \Lie(\cG)^\vee$ is an isomorphism.

Let $\cL$ be an invertible $\cO_S$-module. Let 
\[
\cL^{\otimes -1}:=\cL^\vee=\sHom_{\cO_S}(\cL,\cO_S).
\]
We denote by
\[
\bV_*(\cL)=\Spec(\Sym_{\cO_S}(\cL^{\otimes -1}))
\]
the covariant line bundle associated with $\cL$. It represents the functor over $S$ defined by
\[
T\mapsto (\cL|_T)(T),
\]
so that $\bV_*(\cL)$ admits a natural action of the additive group scheme $\Ga=\Spec(\cO_S[Z])$ over $S$. 
Moreover, for any scheme $T$ over $S$, there exists a $\Ga$-equivariant natural isomorphism of group schemes
\[
\bV_*(\cL|_T)\to \bV_*(\cL)|_T.
\]


We have an isomorphism of $\cO_S$-modules
\[
\cL^{(q)}\to \cL^{\otimes q},\quad 1\otimes l\mapsto l^{\otimes q},
\]
which induces a $\Ga$-equivariant isomorphism of group schemes 
\[
\bV_*(\cL)^{(q)}\simeq \bV_*(\cL^{(q)})\to \bV_*(\cL^{\otimes q}).
\] 
Under this identification, the relative $q$-th power Frobenius homomorphism 
\[
\tau:\bV_*(\cL)\to \bV_*(\cL)^{(q)}\simeq  \bV_*(\cL^{\otimes q})
\] 
is given by the natural inclusion
\[
\Sym_{\cO_S}(\cL^{\otimes -q})\to \Sym_{\cO_S}(\cL^{\otimes -1}).
\]
In particular, for any scheme $T$ over $S$, the map that $\tau$ induces on the $T$-valued points is
\[
\bV_*(\cL)(T)\to\bV_*(\cL^{\otimes q})(T),\quad l\mapsto l^{\otimes q}.  
\]

For any $\bF_q$-module schemes $\cG$ and $\cH$ over $S$, we denote by 
\[
\sHom_{\bF_q,S}(\cG,\cH),\quad \Hom_{\bF_q,S}(\cG,\cH)
\] 
the sheaf and the module of $\bF_q$-linear homomorphisms over $S$. Then we have a natural isomorphism of $\cO_S$-modules
\[
\bigoplus_{m\geq 0} \cL^{\otimes -q^m}\to \sHom_{\bF_q,S}(\bV_*(\cL),\Ga),\quad l\mapsto (Z\mapsto l). 
\]

\begin{dfn}\label{DfnDM}
	Let $r\geq 1$ be an integer and let $S$ be a scheme over $A$. A Drinfeld module of rank $r$ over $S$ is a pair $E=(\cL_E,\Phi^E)$ consisting of an invertible 
	$\cO_S$-module $\cL_E$ and an $\bF_q$-algebra homomorphism
	\[
	\Phi^E:A\to \End_{\bF_q,S}(\bV_*(\cL_E)),\quad a\mapsto \Phi^E_a=\sum_{i=0}^{r\deg(a)}\alpha_i^E(a)\tau^i
	\]
	with $\alpha^E_i(a)\in \cL_E^{\otimes 1-q^i}(S)$ such that $\alpha^E_0(a)\in \cO_S(S)$ agrees with the image of $a$ by the structure map $A\to \cO_S(S)$ and $
	\alpha^E_{r\deg(a)}(a)$ is nowhere vanishing.  
	\end{dfn}
	
	\begin{dfn}\label{DfnMorDM}
	For any Drinfeld modules $E$ and $F$ of some rank over $S$, we define a morphism $E\to F$ of Drinfeld modules over $S$ as a homomorphism of group schemes
	$\bV_*(\cL_E)\to \bV_*(\cL_F)$ over $S$ which is compatible with $A$-actions given by $\Phi^E$ and $\Phi^F$. By \cite[Proposition 3.6]{Pink}, any such morphism 
	admitting a two-sided inverse is compatible with $\Ga$-actions, and thus it is induced by an isomorphism of $\cO_S$-modules $\cL_E\to \cL_F$. We refer to a 
	morphism of Drinfeld modules with a two-sided inverse as an isomorphism.
\end{dfn}

Let $S$ be a scheme over $A$ and let $E$ be a Drinfeld module of rank two over $S$. Let
\[
j_t(E):=\alpha_1^E(t)^{\otimes q+1}\otimes \alpha^E_2(t)^{\otimes -1}\in \cO_S(S).
\]
Then $j_t(E)$ depends only on the isomorphism class of $E$.

\subsection{Drinfeld modular curves}\label{SubsecDMC}

Let $C=\Spec(A[Z])$ be the Carlitz module over $A$. It is the Drinfeld module of rank one with trivial underlying invertible sheaf such that its $A$-action is 
given by
\[
\Phi^C_t=\theta + \tau.
\]
For any $A$-module schemes $\cG$ and $\cH$ over a scheme $S$, we denote by $\Isom_{A,S}(\cG,\cH)$ the set of isomorphisms $\cG\to \cH$ of $A$-module schemes over 
$S$, and by $\sIsom_{A,S}(\cG,\cH)$ the functor over $S$ associating with a scheme $T$ the set
\[
\sIsom_{A,S}(\cG,\cH)(T):=\Isom_{A,T}(\cG|_T,\cH|_T).
\]

Let $\frn\in A\setminus \bF_q$ be a monic polynomial. For a scheme $S$ over $A$ and a Drinfeld module $E$ of some rank over $S$, let
\[
E[\frn]=\Ker(\Phi^E_\frn: E\to E).
\]
It is a finite locally free $A$-module scheme over $S$. When $S$ is over $A[1/\frn]$, it is also \'{e}tale. 
For a finite group $G$, we denote by $\underline{G}$ the constant group scheme associated with $G$.

\begin{dfn}\label{DfnLevelNStr}
	Let $S$ be a scheme over $A[1/\frn]$ and let $E$ be a Drinfeld module of rank two over $S$. A $\Gamma(\frn)$-structure on $E$ is an isomorphism 
	\[
	\alpha: \underline{(A/(\frn))^2}\to E[\frn]
	\]
	of $A$-module schemes over $S$.
\end{dfn}

It is known that the functor over $A[1/\frn]$ associating with a scheme $S$ the set of isomorphism classes of pairs $(E,\alpha)$ of a 
Drinfeld module $E$ of rank two over $S$ and a $\Gamma(\frn)$-structure $\alpha$ on $E$ is representable by an affine scheme $Y(\frn)$ which is smooth of 
relative dimension one over $A[1/\frn]$.  

\begin{dfn}\label{DfnLevelGamma1NStr}
	Let $S$ be a scheme over $A[1/\frn]$ and let $E$ be a Drinfeld module of rank two over $S$. A $\Gamma_1(\frn)$-structure on $E$ is a closed immersion 
	\[
	\lambda: C[\frn]|_S\to E[\frn]
	\]
	of $A$-module schemes over $S$.
\end{dfn}  

By \cite[p.~81]{Ha_DMC}, the functor sending a scheme $S$ over $A[1/\frn]$ to the set of isomorphism classes of pairs $(E,\lambda)$ of a 
Drinfeld module $E$ of rank two over $S$ and a $\Gamma_1(\frn)$-structure $\lambda$ on $E$ is representable by an affine scheme $Y_1(\frn)$ which is smooth of 
relative dimension one over $A[1/\frn]$.

Let $\lambda$ be a $\Gamma_1(\frn)$-structure on $E$. Then the $A$-module scheme $E[\frn]/\Img(\lambda)$ is finite \'{e}tale over $S$ and the functor
\[
\sIsom_{A,S}(\underline{A/(\frn)},E[\frn]/\Img(\lambda))
\]
is represented by a finite \'{e}tale $(A/(\frn))^\times$-torsor $I_{(E,\lambda)}$ over $S$.

Consider the group
\[
\Gamma_1(\frn)=\left\{\gamma=\begin{pmatrix}
	a&b\\c&d
\end{pmatrix}\in \mathit{SL}_2(A)\ \middle|\ \gamma\equiv
\begin{pmatrix}
	1& *\\0&1
\end{pmatrix}\bmod \frn\right\},
\]
which acts on the Drinfeld upper half plane $\Omega$ by $\gamma(z)=\frac{az+b}{cz+d}$. In order to define a Hodge bundle on a model over $A[1/\frn]$ of the 
compactification of the curve 
$\Gamma_1(\frn)\backslash \Omega$, a level structure called a $\Gamma_1^\Delta(\frn)$-structure is studied \cite[\S3]{Ha_DMC}. It is defined as follows.

\begin{dfn}\label{DfnLevelDatum}
	An $h$-level pair $(\frn,\Delta)$ consists of a monic polynomial $\frn\in A\setminus\bF_q$ and a subgroup $\Delta\subseteq (A/
	(\frn))^\times$ satisfying the following conditions.
	\begin{enumerate}
		\item $\frn$ has a prime factor such that the degree of its residue field over $\bF_q$ is prime to $q-1$.
		\item The natural map $\Delta\to (A/(\frn))^\times/\bF_q^\times$ is an isomorphism. 
	\end{enumerate}
	Note that if $\frn$ satisfies the former condition, then there exists $\Delta$ satisfying the latter condition. The $h$ in the term reflects
	its relationship with Gekeler's $h$-function, as we will see in \S\ref{SubsecAD} and \S\ref{SubsecIntrinsic}. 
\end{dfn}

\begin{dfn}\label{DfnLevelGamma1NDeltaStr}
	Let $(\frn,\Delta)$ be an $h$-level pair.
	Let $S$ be a scheme over $A[1/\frn]$ and let $E$ be a 
	Drinfeld module of rank two over $S$. A $\Gamma_1^\Delta(\frn)$-structure on $E$ is a pair $(\lambda, \mu)$ of a $\Gamma_1(\frn)$-structure $\lambda$ on $E$ 
	and an element $\mu\in (I_{(E,\lambda)}/\Delta)(S)$.
\end{dfn}

Then the functor sending a scheme $S$ over $A[1/\frn]$ to the set of isomorphism classes of triples $(E,\lambda,\mu)$ consisting of a 
Drinfeld module $E$ of rank two over $S$ and a $\Gamma_1^\Delta(\frn)$-structure $(\lambda,\mu)$ on $E$ is representable by an affine scheme $Y_1^\Delta(\frn)$ 
which is a finite \'{e}tale $\bF_q^\times$-torsor over $Y_1(\frn)$. In fact, for the universal object $(E_\univ,\lambda_\univ)$ over $Y_1(\frn)$, we have
\[
Y_1^\Delta(\frn)=I_{(E_\univ,\lambda_\univ)}/\Delta.
\] 

Let $\bullet\in\{\emptyset,\Delta\}$. For any scheme $S$ over $A[1/\frn]$ and any Drinfeld module $E$ of rank two over $S$, 
we denote by $[\Gamma_1^\bullet(\frn)]_{E/S}$ be 
the finite \'{e}tale scheme over $S$ 
representing the functor over $S$ associating with a scheme $T$ the set of $\Gamma_1^\bullet(\frn)$-structures on $E|_T$.

\subsection{Tate--Drinfeld modules}\label{SubsecTD}

Here we recall the theory of Tate--Drinfeld modules, following the exposition of \cite[\S2.2]{Boeckle} and \cite[\S4]{Ha_DMC}.

Let $x$ be an indeterminate and let $A((x))$ be the Laurent power series ring over $A$. 
Let
\[
\Lambda=\left\{\Phi_a^C\left(\frac{1}{x}\right)\ \middle|\ a\in A\right\}\subseteq A((x)).
\]
Note that for any nonzero element $a\in A$, we have 
\[
\Phi^C_a\left(\frac{1}{x}\right)\in\frac{1}{x^{q^{\deg(a)}}}A[[x]]^\times.
\] 
Define
\[
e_{\Lambda}(X)=X\prod_{\alpha\in \Lambda\setminus\{0\}}\left(1-\frac{X}{\alpha}\right)\in X+ x X^2 A[[x]][[X]].
\]
Then $e_{\Lambda}(X)$ is entire with respect to the $x$-adic topology on $A[[x]]$. In particular, for any element $\beta$ of a finite extension 
$L/K((x))$, plugging in $X=\beta$ defines an element $e_\Lambda(\beta)\in L$. Moreover, since $\Lambda$ is stable under the action of 
$\bF_q^\times$ defined by $x\mapsto c^{-1}x$, each coefficient of $e_\Lambda(X)$ lies in $A[[x^{q-1}]]$.

We denote by $e_\Lambda^{-1}(X)$ the element of $A[[x^{q-1}]][[X]]$ satisfying $e_\Lambda(e_\Lambda^{-1}(X))=e_\Lambda^{-1}(e_\Lambda(X))=X$.
For any $a\in A$, let
\[
\Phi_a^\Lambda(X):=e_\Lambda(\Phi^C_a(e_\Lambda^{-1}(X)))\in A[[x^{q-1}]][[X]].
\]
Then it is an element of $A[[x^{q-1}]][X]$ having the following description. For $K=\bF_q(t)$ and an algebraic closure $K((x))^\alg$ of the Laurent power series 
field $K((x))$, 
consider the subgroup
\[
(\Phi_a^C)^{-1}(\Lambda)=\{y\in K((x))^\alg\mid \Phi^C_a(y)\in \Lambda\}.
\]
Let $\Sigma_a\subseteq (\Phi_a^C)^{-1}(\Lambda)$ be a complete set of representatives of 
\[
((\Phi^C_a)^{-1}(\Lambda)/\Lambda)\setminus\{0\}.
\]
Since $A$ is a regular domain, the proof of \cite[Proposition 2.9]{Boeckle} is valid in this case and we have
\[
\Phi_a^\Lambda(X)=aX\prod_{\beta\in \Sigma_a}\left(1-\frac{X}{e_\Lambda(\beta)}\right).
\]

Write
\[
\Phi^\Lambda_t=\theta+a_1 \tau+a_2 \tau^2,\quad a_i\in A[[x^{q-1}]].
\]
By \cite[Lemma 4.1]{Ha_DMC}, we have
\begin{equation}\label{EqnTDCoeff}
	a_1\in 1+x^{q-1}A[[x^{q-1}]],\quad a_2\in x^{q-1}A[[x^{q-1}]]^\times.
\end{equation}

\begin{dfn}\label{DfnTD}
	The Tate--Drinfeld module $\TD(\Lambda)$ is the Drinfeld module of rank two over $A((x))$ with trivial underlying 
	invertible sheaf defined by $
	\Phi^{\TD(\Lambda)}_t=\Phi^\Lambda_t$. 
	By (\ref{EqnTDCoeff}), it is obtained by
	the base extension of a Drinfeld module $\TD^\td(\Lambda)$ over $A((x^{q-1}))$ with trivial underlying invertible sheaf 
	along the natural map $A((x^{q-1}))\to A((x))$.
	
	For any $A$-algebra $B$, let 
	\[
	\TD(\Lambda)_{/B}:=\TD(\Lambda)|_{B((x))},\quad \TD^\td(\Lambda)_{/B}:=\TD^\td(\Lambda)|_{B((x^{q-1}))}
	\]
	be the base extensions along the natural maps $A((x))\to B((x))$ and $A((x^{q-1}))\to B((x^{q-1}))$, respectively.
\end{dfn}

\begin{rmk}
	In \cite{Ha_DMC}, it is assumed that the base ring $R_0$ is flat over $A[1/\frn]$. The flatness assumption is only used to assure the flatness of 
	the map $A\to R_0$ in the proof of \cite[(4.5)]{Ha_DMC}.
	We can drop the assumption by proving \cite[(4.5)]{Ha_DMC} over $A$ as above and then extend the base ring to $R_0$.
\end{rmk}

\section{Dual Drinfeld modules and the Kodaira--Spencer map}\label{SecKS}

\subsection{Biderivations and the de Rham module}\label{SubsecBider}

In this subsection, we recall the theory of biderivations due to Gekeler \cite[\S3]{Gek_dR}, following \cite[\S2]{PapRam}. 

Let $S$ be a scheme over $A$. Let $E$ be a Drinfeld module of rank two over $S$. Let $G=\Ga$ or $C$.

\begin{dfn}\label{DfnBider}
	\begin{enumerate}
		\item An $(E,G)$-biderivation over $S$ is an $\bF_q$-linear homomorphism 
		\[
		\delta: A\to \Hom_{\bF_q,S}(E,G),\quad a\mapsto \delta_a
		\] 
		satisfying the following conditions.
		\begin{itemize}
			\item $\delta_{\lambda}=0$ for any $\lambda\in \bF_q$.
			\item $\delta_{ab}=\Phi_a^G\circ \delta_b+\delta_a\circ \Phi_b^E$ for any $a,b\in A$.
		\end{itemize}
		The $\bF_q$-vector space of $(E,G)$-biderivations over $S$ is denoted by $\Der(E,G)$.
		
		\item The $\bF_q$-vector space of $(E,G)$-biderivations $\delta$ such that the map $\delta_t^*: \omega_G\to \omega_E$ 
		is zero is denoted by $\Der_0(E,G)$.
		
		\item An $(E,G)$-biderivation $\delta$ is said to be inner if there exists an element $f\in \Hom_{\bF_q,S}(E,G)$ satisfying
		$\delta=\delta_f$, where $\delta_f$ is the $(E,G)$-biderivation defined by
		\[
		(\delta_f)_a=f\circ \Phi_a^E-\Phi_a^G\circ f
		\]
		for any $a\in A$. The $\bF_q$-vector space of inner $(E,G)$-biderivations over $S$ is denoted by $
		\Der_\inner(E,G)$.
		
		\item An inner $(E,G)$-biderivation $\delta_f$ is said to be strictly inner if the map $f^*:\omega_G\to \omega_E$ is zero. The $\bF_q$-vector space of 
		strictly inner $(E,G)$-biderivations is 
		denoted by $\Der_\si(E,G)$.
	\end{enumerate}
\end{dfn}

The $\bF_q$-vector space $\Der(E,G)$ admits right and left $A$-actions defined by
\[
(\delta*c)_a=\delta_a\circ \Phi_c^E,\quad (c*\delta)_a=\Phi_c^G\circ \delta_a
\]
for any $c\in A$.
Then the $\bF_q$-subspaces of $\Der(E,G)$
\[
\Der_0(E,G)\supseteq \Der_\inner(E,G)\supseteq \Der_\si(E,G) 
\]
are stable under these $A$-actions, and two actions agree with each other on 
the quotient $\Der(E,G)/\Der_\inner(E,G)$ \cite[p.~412]{PapRam}. 

Note that we have an $\bF_q$-linear isomorphism
\begin{equation}\label{EqnDfnEVt}
	\Der(E,G)\to \Hom_{\bF_q,S}(E,G),\quad \delta\mapsto \delta_t.
\end{equation}
We identify these $\bF_q$-vector spaces via this isomorphism. Moreover, writing $G=\Spec(\cO_S[Z])$, we have an isomorphism of $\cO_S$-modules
\begin{equation}\label{EqnIdentfyHom}
	\bigoplus_{m\geq 0}\cL^{\otimes -q^m}_E\to \sHom_{\bF_q,S}(E,G),\quad l\mapsto (Z\mapsto l).
\end{equation}

Suppose that $S=\Spec(B)$ for an $A$-algebra $B$.
When $G=\Ga$, the linear action of $B$ on $\Ga$ gives structures of $B$-modules on $\Der(E,\Ga)$ and its $\bF_q$-subspaces defined above, which are compatible with 
the
left $A$-action $\delta\mapsto c*\delta$.

\begin{dfn}\label{DfnHdR}
	We call the $B$-module
	\[
	\DR(E,\Ga):=\Der_0(E,\Ga)/\Der_\si(E,\Ga)
	\]
	the de Rham module of $E$. We denote by $\cH_{\dR}(E)$ the $\cO_S$-module associated with $\DR(E,\Ga)$ and refer to it as the 
	de Rham sheaf of $E$.
\end{dfn}
Under the identification via the isomorphisms (\ref{EqnDfnEVt}) and (\ref{EqnIdentfyHom}),
we have isomorphisms of $\bF_q$-vector spaces
\begin{equation}\label{EqnIdentifyDer}
\begin{aligned}
	\bigoplus_{m\geq 1}\cL^{\otimes -q^m}_E(S)&\to \Der_0(E,G),\\ 
	\cL_E^{\otimes-q}(S)&\to \Der_0(E,G)/\Der_\inner(E,G),\\
		\cL_E^{\otimes -q}(S)\oplus \cL_E^{\otimes-q^2}(S)&\to \Der_0(E,G)/\Der_\si(E,G)
\end{aligned}
\end{equation}
and also an isomorphism of $\cO_S$-modules
\begin{equation}\label{EqnIdentifyDR}
\cL_E^{\otimes-q}\oplus \cL_E^{\otimes-q^2}\to \cH_{\dR}(E).
\end{equation}

Let $R\to B$ be a homomorphism of $A$-algebras and let $D\in \Der_R(B)$. 
When the invertible sheaf $\cL_E$ is trivial, write $\cL_E=\cO_S l$ with some nowhere vanishing section $l\in \cL_E(S)$. Then 
the derivation $D$ acts $R$-linearly on $\Der_0(E,G)$ by
\[
D(b_1 l^{\otimes -q}+\cdots + b_i l^{\otimes -q^i}):=D(b_1) l^{\otimes -q}+\cdots + D(b_i) l^{\otimes -q^i}\quad (b_i\in B).
\]
Since $D(b^q)=0$ for any $b\in B$, this is independent of the choice of a generator $l$. By gluing, we obtain an $R$-linear map
$\nabla_D: \Der_0(E,G)\to \Der_0(E,G)$. 
For $G=\Ga$, the map $\nabla_D$ preserves the $\bF_q$-subspace $\Der_\si(E,\Ga)$.
Thus we obtain an $R$-linear map
\[
\nabla_D: \DR(E,\Ga)\to \DR(E,\Ga).
\]

\subsection{Dual Drinfeld modules}\label{SubsecDualDM}

Let $B$ be an $A$-algebra and let $S=\Spec(B)$. 
Let $E=(\cL_E,\Phi^E)$ be a Drinfeld module of rank two over $B$. Write
\[
\Phi_t^E=\theta + \alpha_1\tau +\alpha_2\tau^2,\quad \alpha_i\in \cL_E^{\otimes 1-q^i}(S).
\]
Then Taguchi defined a dual $E^D$ of $E$, which is a Drinfeld module of rank two over $B$
\cite[p.~578]{Tag}. In fact, Papanikolas--Ramachandran \cite[p.~415]{PapRam} showed that the dual $E^D$ represents the functor
\[
(B\text{-algebras})\to (A\text{-modules}),\quad B'\mapsto \Der_0(E|_{B'},C|_{B'})/\Der_\inner(E|_{B'},C|_{B'}).
\]
Then (\ref{EqnIdentifyDer}) gives an isomorphism of $\bF_q$-vector space schemes over $B$
\begin{equation*}
\bV_*(\cL^{\otimes -q}_E)\to E^D.
\end{equation*}
Under this isomorphism, the $A$-action on $E^D$ is given by 
\begin{equation}\label{EqnTmodStrOnDualDM}
\Phi^{E^D}_t=\theta- \alpha_1\otimes \alpha_2^{\otimes -1}\tau+ \alpha_2^{\otimes -q}\tau^2.
\end{equation}

By \cite[Lemma 2.21]{Ha_HTT}, we have isomorphisms of $B$-modules
\begin{equation}\label{EqnHodgeMaps}
	\begin{aligned}
\Der_\inner(E,\Ga)/\Der_\si(E,\Ga)&\to \Lie(E)^\vee=\omega_E,\\ \Lie(E^D)&\to \Der_0(E,\Ga)/\Der_\inner(E,\Ga).
\end{aligned}
\end{equation}
Write $\Ga=\Spec(B[Z])$. Then the former is induced by composing the map
\[
\Der_\inner(E,\Ga)\to \Hom_{B}(\Lie(E),\Lie(\Ga)),\quad \delta_f\mapsto \Lie(f)
\]
and the isomorphism $dZ:\Lie(\Ga)\to B$. For the latter, consider the scheme of dual numbers $S_\vep=\Spec(B[\vep]/(\vep^2))$.
We have
\[
\Lie(E^D)=\Ker(E^D(S_\vep)\to E^D(S))
\]
and any element $\delta\in \Der_0(E|_{S_\vep},C|_{S_\vep})$ can be written as
$\delta=\delta^0+\vep \delta^1$ 
with $\delta^0\in \Der_0(E,C)$ and $\delta^1\in \Der_0(E,\Ga)$. Then the latter map of (\ref{EqnHodgeMaps}) is induced by $\delta\mapsto \delta^1$.

The maps of (\ref{EqnHodgeMaps}) yield an exact sequence of $B$-modules
\begin{equation}\label{EqnHodgeFil}
	\xymatrix{
		0 \ar[r] & \omega_E \ar[r]& \DR(E,\Ga) \ar[r]& \Lie(E^D) \ar[r]& 0,
	}
\end{equation}
which is split since the $B$-module $\Lie(E^D)$ is locally free.

For any homomorphism of $A$-algebras $R\to B$ and $D\in \Der_R(B)$, the maps of (\ref{EqnHodgeMaps}) give the $R$-linear map
\[
\pi_D: \omega_E\to \DR(E,\Ga)\overset{\nabla_D}{\to} \DR(E,\Ga)\to \Lie(E^D),
\]
which is in fact $B$-linear. Hence this yields a $B$-linear map
\[
\KS_{E/B/R}:\Der_R(B)\to \Hom_B(\omega_E,\Lie(E^D))=\omega_E^\vee\otimes_B \Lie(E^D)
\]
which is called the Kodaira--Spencer map for $E$ \cite[(6.6)]{Gek_dR}. 

Suppose that the $B$-module $\Omega^1_{B/R}$ is finite locally free. Then, by double duality, we have a natural isomorphism of $B$-modules
\[
\Omega^1_{B/R}\to \Hom_B(\Der_R(B),B).
\]
Hence, by taking the $B$-linear dual, the map $\KS_{E/B/R}$ defines a $B$-linear map
\begin{equation*}
\KS^\vee_{E/B/R}: \omega_E\otimes_B \omega_{E^D}\to \Omega_{B/R}^1.
\end{equation*}
Moreover, let $B'$ be a formally \'{e}tale $B$-algebra, so that the natural map $B'\otimes_B \Omega^1_{B/R}\to \Omega^1_{B'/R}$
is an isomorphism. This induces an isomorphism of $B'$-modules
\[
\Der_R(B')\to B'\otimes_B\Der_R(B),\quad D|_{B'}\mapsto 1\otimes D
\]
satisfying $(D|_{B'})(b)=D(b)|_{B'}$ for any $b\in B$.
Then we have the natural commutative diagram
\begin{equation}\label{EqnDiagKSBCB}
\xymatrix{
	\omega_E|_{B'}\otimes_{B'} \omega_{E^D}|_{B'}\ar[rr]^-{1\otimes\KS^\vee_{E/B/R}}\ar[d]_{\wr}& &B'\otimes_B\Omega_{B/R}^1\ar[d]^{\wr}\\
	\omega_{E|_{B'}}\otimes_{B'} \omega_{E|_{B'}^D}\ar[rr]_-{\KS^\vee_{E|_{B'}/B'/R}}& &\Omega_{B'/R}^1,
}
\end{equation}
where the vertical arrows are isomorphisms. 

Similarly, for any homomorphism $R\to R'$ of $A$-algebras and $B_{R'}:=R'\otimes_R B$, we have a natural isomorphism of $B_{R'}$-modules
$R'\otimes_R\Omega^1_{B/R}\to \Omega^1_{B_{R'}/R'}$. Then this induces the natural commutative diagram with isomorphic vertical arrows
\begin{equation}\label{EqnDiagKSBCR}
\xymatrix{
	\omega_E|_{B_{R'}}\otimes_{B_{R'}} \omega_{E^D}|_{B_{R'}}\ar[rr]^-{1\otimes\KS^\vee_{E/B/R}}\ar[d]_{\wr}& &B_{R'}\otimes_{B}\Omega_{B/R}^1
	&R'\otimes_R\Omega_{B/R}^1\ar[d]^{\wr}\ar[l]_-{\sim} \\
	\omega_{E|_{B_{R'}}}\otimes_{B_{R'}} \omega_{E|_{B_{R'}}^D}\ar[rrr]_-{\KS^\vee_{E|_{B_{R'}}/B_{R'}/R'}}& & &\Omega_{B_{R'}/R'}^1.
}
\end{equation}

\subsection{Case of trivial underlying invertible sheaf}\label{SubsecTrivCase}

Let us give an explicit description of the Hodge filtration (\ref{EqnHodgeFil}) when $\cL_E$ is trivial and $S=\Spec(B)$. 
Let $l\in\cL_E(S)$ be a nowhere vanishing section, so that we have $\cL_E=\cO_S l$. 
Then $\cL_E^{\otimes -1}=\cO_Sl^{\otimes -1}$ with the dual basis $l^{\otimes -1}$ of $l$.
We have an isomorphism of $\cO_S$-algebras
\[
\cO_S[X]\to \Sym_{\cO_S}(\cL_E^{\otimes -1}),\quad X\mapsto l^{\otimes -1}.
\]
Similarly, we also have an isomorphism of $\cO_S$-algebras
\[
\cO_S[Y]\to \Sym_{\cO_S}(\cL_E^{\otimes q}),\quad Y\mapsto l^{\otimes q}.
\]
They induce $\Ga$-equivariant isomorphisms of $\bF_q$-vector space schemes over $S$
\begin{equation}\label{EqnIdentificationXY}
E=\bV_*(\cL_E)\to \Spec(\cO_S[X]),\quad E^D=\bV_*(\cL_E^{\otimes -q})\to \Spec(\cO_S[Y]).
\end{equation}
Using these maps, we identify as $E=\Spec(\cO_S[X])$ and $E^D=\Spec(\cO_S[Y])$, respectively. Write
\[
\alpha_i=\mathbf{a}_i l^{\otimes 1-q^i}\quad (\mathbf{a}_i\in B).
\]
Then the isomorphisms (\ref{EqnIdentificationXY}) enable us to identify $\Phi^E_t$ and $\Phi^{E^D}_t$ as
\[
\Phi_t^E=\theta+ \mathbf{a}_1\tau+ \mathbf{a}_2\tau^2,\quad \Phi_t^{E^D}=\theta-\mathbf{a}_1\mathbf{a}_2^{-1}\tau+\mathbf{a}_2^{-q} \tau^2.
\]

For $G=\Ga$ or $C$, write $G=\Spec(\cO_S[Z])$. Under the isomorphism (\ref{EqnIdentfyHom}), the relative $q$-th power Frobenius map
\[
\tau=(Z\mapsto X^q)\in \Hom_{\bF_q,S}(E,G)\simeq \bigoplus_{m\geq 0}\cL_E^{\otimes -q^m}(S)
\]
corresponds to $l^{\otimes -q}$. 
Since (\ref{EqnIdentificationXY}) yields $dX=l^{\otimes -1}$ and $\frac{d}{dY}=l^{\otimes -q}$,
we may write
\begin{equation}\label{EqnIdentificationHodgeFil}
\omega_E=\cO_S dX,\quad \Lie(E^D)=\cO_S\frac{d}{dY},\quad \cH_\dR(E)=\cO_S\tau\oplus \cO_S\tau^2.
\end{equation}
Then the exact sequence (\ref{EqnHodgeFil}) equals
\[
\xymatrix{
0 \ar[r] & B dX \ar[r]^-{i} &  B\tau\oplus B\tau^2\ar[r]^-{\pi} & B\frac{d}{dY} \ar[r] &  0.
}
\]

By the description of (\ref{EqnHodgeMaps}), the maps $i$ and $\pi$ in this sequence are given as follows.
Consider the map
\[
\id: E\to \Ga=\Spec(B[Z]),\quad Z\mapsto X.
\] 
Then $i(dX)$ equals the class of $\delta_{\id}$, where $\delta_\id$ is the inner biderivation satisfying $(\delta_\id)_t=\id\circ \Phi_t^E-\Phi_t^{\Ga}\circ 
\id$.
Namely, we have
\begin{equation}\label{EqnEdX}
i(dX)=\mathbf{a}_1\tau+\mathbf{a}_2\tau^2.
\end{equation}

On the other hand, let $b_1,b_2\in B$ and $\delta=b_1\tau+b_2\tau^2$. Let $S_\vep=\Spec(B[\vep]/(\vep^2))$. Then the element
\[
\pi(\delta)\in \Lie(E^D)=\Ker(E^D(S_\vep)\to E^D(S))
\]
is given by the class of the biderivation $(t\mapsto \vep\delta)\in \Der_0(E|_{S_\vep},C|_{S_\vep})$ modulo $\Der_\inner(E|_{S_\vep},C|_{S_\vep})$. Since
\[
\vep \mathbf{a}_2^{-1}b_2(\theta+\mathbf{a}_1\tau+ \mathbf{a}_2\tau^2)-(\theta+\tau)(\vep \mathbf{a}_2^{-1}b_2)=\vep(\mathbf{a}_2^{-1}\mathbf{a}
_1b_2\tau+b_2\tau^2),
\]
the class agrees with that of $\vep(b_1-\mathbf{a}_2^{-1}\mathbf{a}_1b_2)\tau$. Since this corresponds to the element 
\[
\vep\mathbf{a}_2^{-1}(\mathbf{a}_2b_1-\mathbf{a}_1b_2) l^{\otimes -q}\in \bV_*(\cL_E^{\otimes -q})(S_\vep)=E^D(S_\vep),
\]
under the identification (\ref{EqnIdentificationXY}) we obtain
\begin{equation}\label{EqnELie}
\pi(\delta)=\mathbf{a}_2^{-1}(\mathbf{a}_2b_1-\mathbf{a}_1b_2)  \frac{d}{dY}.
\end{equation}

\section{Extension to the compactification}\label{SecExt}

Let $(\frn,\Delta)$ be an $h$-level pair. Let $R$ be an $A[1/\frn]$-algebra which is an excellent regular domain. 

\subsection{Compactification of Drinfeld modular curves}\label{SubsecCompactify}

First we recall the definition of the compactification $X^\Delta_1(\frn)_R$ of $Y_1^\Delta(\frn)_R$ and its 
description around cusps, following 
\cite[p.~85 and \S5]{Ha_DMC}.

Let $E_\univ^\Delta=(\cL^\Delta_\univ,\Phi^{E^\Delta_\univ})$ 
be the universal Drinfeld module over $Y_1^\Delta(\frn)_R$, which is the base extension of the universal Drinfeld module $E_\univ$ 
over $Y_1(\frn)_R$ by the natural projection $Y_1^\Delta(\frn)_R\to Y_1(\frn)_R$. Consider the morphism
\begin{equation*}
j_t: Y_1^\Delta(\frn)_R\to \bA^1_{R}=\Spec(R[j]),\quad j\mapsto j_t(E_\univ^\Delta).
\end{equation*}
We define the compactification $X_1^\Delta(\frn)_R$ of $Y_1^\Delta(\frn)_R$ by the normalization of $\bP^1_{R}$ in $Y_1^\Delta(\frn)_R$ along the 
morphism $j_t$. Then it is known \cite[p.~85]{Ha_DMC} that the morphism $X_1^\Delta(\frn)_R\to \Spec(R)$ is smooth of relative dimension one with geometrically 
connected fibers.
Moreover, the smoothness implies that the formation of $X_1^\Delta(\frn)_R$ is compatible with any base extension $R\to R'$ of excellent regular domains over $A[1/
\frn]$.

Define a scheme $\widehat{\Cusps}^\Delta_R$ by the cartesian diagram
\begin{equation}\label{EqnDiagDfnCusps}
\xymatrix{
	\widehat{\Cusps}^\Delta_R\ar[r]\ar[d] & X_1^\Delta(\frn)_R\ar[d]\\
	\Spec(R[[1/j]])\ar[r] & \bP^1_R.
}
\end{equation}
Since (\ref{EqnTDCoeff}) implies $j_t(\TD^\td(\Lambda)_{/R})\in x^{1-q}R[[x^{q-1}]]^\times$, we have an isomorphism 
\[
y^\td: S_{0,R}:=\Spec(R((x^{q-1})))\to \Spec(R((1/j))),\quad j\mapsto j_t(\TD^\td(\Lambda)_{/R}).
\]
Then we can show that $\widehat{\Cusps}^\Delta_R$ is isomorphic to the normalization of $\Spec(R[[x^{q-1}]])$ in a scheme $Y_1^\Delta(\frn)_{S_{0,R}}$ defined by 
the following cartesian diagram
\begin{equation*}
\xymatrix{
	Y_1^\Delta(\frn)_{S_{0,R}} \ar[rr]\ar[d]& & Y_1^\Delta(\frn)_{R}\ar[d]\\
	S_{0,R}\ar[r]_-{y^\td}^-{\sim} &\Spec(R((1/j)))\ar[r] & \bA^1_R.
}
\end{equation*}

For $\bullet\in \{\emptyset,\Delta\}$, write 
\[
\begin{aligned}
 [\Gamma_1^\bullet(\frn)]_{\TD^\td/R}&:=[\Gamma_1^\bullet(\frn)]_{\TD^\td(\Lambda)_{/R}/\Spec(R((x^{q-1})))},\\
 [\Gamma_1^\bullet(\frn)]_{\TD/R}&:=[\Gamma_1^\bullet(\frn)]_{\TD(\Lambda)_{/R}/\Spec(R((x)))}.
\end{aligned}
\]
By \cite[Lemma 3.4]{Ha_DMC}, we have
\begin{equation}\label{EqnAutTD}
\bF_q^\times=\Aut_{A,R((x^{q-1}))}(\TD^\td(\Lambda)_{/R}),
\end{equation}
where the symbol on the right-hand side denotes the group of $A$-linear automorphisms over $R((x^{q-1}))$.
This group acts
on $[\Gamma_1(\frn)]_{\TD^\td/R}$ and $[\Gamma_1^\Delta(\frn)]_{\TD^\td/R}$ 
by
\begin{align*}
 [c](\TD^\td(\Lambda)|_T,\lambda)&:=(\TD^\td(\Lambda)|_T,c\lambda),\\
 [c](\TD^\td(\Lambda)|_T,\lambda,\mu)&:=(\TD^\td(\Lambda)|_T,c\lambda,c\mu)
\end{align*}
for any scheme $T$ over $S$ and any $c\in \bF_q^\times$.

Moreover, for any $c\in \bF_q^\times$, let
\[
g_c: R((x))\to R((x)),\quad x\mapsto c^{-1}x.
\]
This defines an action of $\bF_q^\times$ on $\Spec(R((x)))$ over $R((x^{q-1}))$.
By \cite[Lemma 5.1]{Ha_DMC}, we have an $\bF_q^\times$-equivariant isomorphism of schemes over $R((x^{q-1}))$
\begin{equation}\label{EqnIsomDeltaDeprive}
[\Gamma_1(\frn)]_{\TD/R}=[\Gamma_1(\frn)]_{\TD^\td/R}\times_{R((x^{q-1}))}\Spec(R((x)))\to [\Gamma_1^\Delta(\frn)]_{\TD^\td/R},
\end{equation}
where $\bF_q^\times$ acts on the source diagonally.

The natural morphism 
\[
T^\td:[\Gamma_1^\Delta(\frn)]_{\TD^\td/R}\to Y_1^\Delta(\frn)_{R}
\] 
is $\bF_q^\times$-invariant, and it
induces
an isomorphism of schemes over $R((x^{q-1}))$
\[
[\Gamma_1^\Delta(\frn)]_{\TD^\td/R}/\bF_q^\times\to Y_1^\Delta(\frn)_{S_{0,R}}.
\]

Let $\cZ_R^\Delta$ be the normalization of $\Spec(R[[x^{q-1}]])$ in $[\Gamma_1^\Delta(\frn)]_{\TD^\td/R}$. Then the $\bF_q^\times$-action on $
[\Gamma_1^\Delta(\frn)]_{\TD^\td/R}$ extends to $\cZ_R^\Delta$. By \cite[Lemma 5.2]{Ha_DMC}, the action is free, and the isomorphism above induces a natural 
isomorphism over 
$R[[x^{q-1}]]$
\begin{equation}\label{EqnIsomCuspsQuotB}
\cZ_R^\Delta/\bF_q^\times\to \widehat{\Cusps}^\Delta_R.
\end{equation}
By composing the natural projection, we obtain a morphism
\[
\sT^\td: \cZ_R^\Delta\to\cZ^\Delta_R/\bF_q^\times\to \widehat{\Cusps}^\Delta_R\to X_1^\Delta(\frn)_R.
\]

We denote by $\Cusps_R^\Delta$ the closed subscheme of $\widehat{\Cusps}_R^\Delta$ defined by $1/j=0$ with the reduced structure. 
By \cite[Theorem 6.3]{Ha_DMC}, we know that $\Cusps_R^\Delta$ is a closed subscheme of $X_1^\Delta(\frn)_R$ which is finite \'{e}tale over $R$ and thus it
defines an effective Cartier divisor on $X_1^\Delta(\frn)_R$. 

\begin{lem}\label{LemBCCusps}
	Let $R\to R'$ be a homomorphism of excellent regular domains over $A[1/\frn]$. Then we have natural isomorphisms
	\[
	\widehat{\Cusps}_{R'}^\Delta\to \widehat{\Cusps}_R^\Delta\,\widehat{\times}_R\,\Spec(R'),\quad \Cusps_{R'}^\Delta\to\Cusps_R^\Delta \times_R\Spec(R'),
	\]
	where $\widehat{\Cusps}_R^\Delta\,\widehat{\times}_R\,\Spec(R')$ denotes the affine scheme defined by the $(1/j)$-adic completion of the affine ring of
	$\widehat{\Cusps}_R^\Delta\times_R\Spec(R')$. 
\end{lem}
\begin{proof}
	This follows in the same way as \cite[Proposition 8.6.6]{KM}. Since $\widehat{\Cusps}_{R}^\Delta$ is finite over $\Spec(R[[1/j]])$, the former assertion 
	follows from \cite[Exercise 0.7.9]{FK} and the cartesian diagrams
	\[
	\xymatrix{
		\widehat{\Cusps}_{R'}^\Delta \ar[r]\ar[d]& \widehat{\Cusps}_{R}^\Delta\times_R \Spec(R')  \ar[r]\ar[d]& X_1^\Delta(\frn)_{R'}\ar[d]\\
		\Spec(R'[[1/j]])\ar[r] & \Spec(R[[1/j]])\times_R\Spec(R')\ar[r] & \bP^1_{R'}.
	}
	\]
	Since the closed subscheme $\Cusps_R^\Delta\times_R\Spec(R')$ of
	$X_1^\Delta(\frn)_{R'}\simeq X_1^\Delta(\frn)_{R}\times_R\Spec(R')$ is finite \'{e}tale over $R'$, it is reduced. Thus it agrees with
	$\Cusps_{R'}^\Delta$ as a subscheme and the latter assertion follows.
\end{proof}

\subsection{Modification at cusps and an extension of the de Rham sheaf}\label{SubsecModify}

Write 
\[
\cP^\Delta_R:=[\Gamma_1^\Delta(\frn)]_{\TD^\td/R}=\Spec(B^\Delta_R),\quad \cZ_R^\Delta=\Spec(\sB^\Delta_R),
\]
where $B^\Delta_R$ is a finite \'{e}tale $R((x^{q-1}))$-algebra and $\sB^\Delta_R$ is a finite torsion-free $R[[x^{q-1}]]$-subalgebra of $B^\Delta_R$.
These rings are equipped with natural actions of the group $\bF_q^\times$ via (\ref{EqnAutTD}) which are compatible with each other.
We denote the actions of $c\in \bF_q^\times$ on these rings by $[c]$.

Let $\cF$ be a finite locally free sheaf on $Y_1^\Delta(\frn)_R$. 
Let $M$ be the finite locally free $B^\Delta_R$-module associated with $(T^\td)^*\cF$. Then the $B^\Delta_R$-module $M$ is equipped with an $\bF_q^\times$-action 
covering 
the action on 
$B^\Delta_R$. For $c\in \bF_q^\times$, we denote the action by
\[
\theta_c: [c]^*M\to M.
\]

\begin{dfn}\label{DfnModification}
	Let $\cF$ be a finite locally free sheaf on $Y_1^\Delta(\frn)_R$. 
	Let $M$ be the finite locally free $B^\Delta_R$-module associated with $(T^\td)^*\cF$. A modification of $\cF$ at cusps is a finite 
	locally free $\sB^\Delta_R$-submodule 
	$\sM$ of $M$ 
	such that $\sM\otimes_{\sB_R^\Delta}B_R^\Delta=M$ and for any $c\in \bF_q^\times$, the map $\theta_c$ sends $[c]^*\sM$ to $\sM$. It follows that the induced 
	map
	$\theta_c:[c]^*\sM\to \sM$ is an isomorphism and defines an action of $\bF_q^\times$ covering that on $\sB^\Delta_R$.
\end{dfn}

Let $\cC_R$ be the category defined as follows.

\begin{itemize}
	\item An object of $\cC_R$ is a pair $(\cF,\sM)$ of a finite locally free sheaf $\cF$ on $Y_1^\Delta(\frn)_R$ and its modification at cusps $\sM$.
	\item A morphism $(\cF,\sM)\to (\cF',\sM')$ is a pair $(f,g)$ of a morphism of $\cO_{Y_1^\Delta(\frn)_R}$-modules $f:\cF\to \cF'$ and 
	a $\sB^\Delta_R$-linear map
	$g:\sM\to \sM'$ satisfying $((T^\td)^*f)|_{\sM}=g$. Then it follows that $g$ is $\bF_q^\times$-equivariant.
\end{itemize}
We say that a sequence $(\cF,\sM)\to (\cF',\sM')\to (\cF'',\sM'')$ in this category is exact if the sequences $\cF\to \cF'\to\cF''$ and $\sM\to \sM'\to \sM''$ are 
exact.

For any homomorphism $R\to R'$ of excellent regular domains over $A[1/\frn]$, we have a base extension functor $\cC_R\to \cC_{R'}$ defined by
\[
(\cF,\sM)\mapsto (\cF|_{Y_1^\Delta(\frn)_{R'}},\sM\otimes_{B_R^\Delta}B_{R'}^\Delta).
\]

\begin{lem}\label{LemExt}
	The functor from the category of finite locally free sheaves on $X_1^\Delta(\frn)_R$ to $\cC_R$ defined by
	\[
	\cF\mapsto (\cF|_{Y_1^\Delta(\frn)_R},(\sT^\td)^*\cF)
	\]
	is an exact equivalence of categories with an exact quasi-inverse. Moreover, this functor is compatible with the base extension along any homomorphism $R\to 
	R'$ of excellent regular 
	domains 
	over $A[1/\frn]$.
\end{lem}
\begin{proof}
	Let us prove the first assertion.
	By \cite[\S12, Theorem 1 (B)]{Mumford}, the category $\cC_R$ is equivalent to the category $\cC'_R$ 
	of pairs $(\cF,\cG)$ of a finite locally free sheaf $\cF$ on $Y_1^\Delta(\frn)_R$
	and a finite locally free sheaf $\cG$ on $\widehat{\Cusps}_R^\Delta$ which agree with each other on $Y_1^\Delta(\frn)_{S_{0,R}}$. 
	Since the equivalence is given by 
	the pull-back along the quotient map by a free action of $\bF_q^\times$, it is exact. 
	Moreover, since the order of the group $\bF_q^\times$ is invertible in $R$ and the quasi-inverse is given by taking the 
	$\bF_q^\times$-invariant submodule, the latter is also 
	exact. Thus we reduce ourselves to showing that the functor
	\begin{equation}\label{EqnLemExtFunctorDash}
	\cF\mapsto (\cF|_{Y_1^\Delta(\frn)_R},\cF|_{\widehat{\Cusps}_R^\Delta})
	\end{equation}
	from the category of finite locally free sheaves on $X_1^\Delta(\frn)_R$ to $\cC'_R$ is an exact equivalence of categories with an exact quasi-inverse. 
	
	Consider the finite map $j_t:X_1^\Delta(\frn)_R\to \bP^1_R$ and the open subscheme $\Spec(R[1/j])$ of $\bP^1_R$. Write the restriction of $X_1^\Delta(\frn)_R$ 
	to this 
	open subscheme by $\Spec(\bB)$. Since the morphism $j_t:Y_1^\Delta(\frn)_R\to\Spec(R[j])$ is flat, the element $1/j\in \bB$ is a nonzero divisor. 
	Furthermore, the finiteness of $j_t$ implies that the affine ring of $\widehat{\Cusps}^\Delta_R$ is identified with 
	the $(1/j)$-adic completion $\hat{\bB}$ of $\bB$ and that of 
	$Y_1^\Delta(\frn)_{S_{0,R}}$ with $\hat{\bB}[1/(1/j)]$. Thus the descent lemma of Beauville--Laszlo \cite[Th\'{e}or\`{e}me]{BeauLasz} shows that the functor 
	(\ref{EqnLemExtFunctorDash}) is 
	an equivalence of categories.
	
	For the exactness, since we are assuming that $R$ is excellent (in particular, Noetherian), the map $\bB\to \hat{\bB}$ is flat and the functor is exact. The 
	exactness of the quasi-inverse follows from \cite[(15.91.16.1)]{Stacks} and the snake lemma. The compatibility with base extension follows from the commutative 
	diagram
	\[
	\xymatrix{
		\cZ_{R'}^\Delta \ar[r]\ar[d] & X_1^\Delta(\frn)_{R'}\ar[d]& Y_1^\Delta(\frn)_{R'}\ar[l]\ar[d]\\
			\cZ_{R}^\Delta \ar[r] & X_1^\Delta(\frn)_{R} & Y_1^\Delta(\frn)_{R}.\ar[l]
	}
	\]
\end{proof}

Define the Hodge bundle on $Y_1^\Delta(\frn)_R$ by
\[
\omega^\Delta_\univ:=\Cot(E^\Delta_\univ)=(\cL^\Delta_\univ)^\vee.
\]
Its extension to $X_1^\Delta(\frn)_R$ is defined in \cite[Theorem 5.3 (2)]{Ha_DMC}. Here we briefly recall the construction, using Lemma \ref{LemExt}.
Over the scheme $\cP_R^\Delta=\Spec(B_R^\Delta)$, we have the canonical isomorphism
\begin{equation}\label{EqnEunivPullBackP}
	\TD^\td(\Lambda)|_{B_R^\Delta}\to E_\univ^\Delta|_{\cP_R^\Delta}.
\end{equation}
Moreover, under this isomorphism, the pull-back of the Hodge bundle to $\cP_R^\Delta$
agrees with the coherent sheaf associated with
\[
\Cot(\TD^\td(\Lambda)|_{B_R^\Delta})=B_R^\Delta dX, 
\]
and the action of 
$c\in 
\bF_q^\times$ via (\ref{EqnAutTD}) is given by $dX\mapsto cdX$ covering the action on $B^\Delta_R$. 
Hence its $\sB_R^\Delta$-submodule $\sB_R^\Delta dX$ is a modification of $\omega^\Delta_\univ$ at cusps, and thus Lemma \ref{LemExt} yields
an invertible sheaf $\bar{\omega}^\Delta_{\univ}$ on $X_1^\Delta(\frn)_R$ extending $\omega^\Delta_\univ$.


Next we define an extension of the de Rham sheaf
\[
\cH_{\dR,\univ}^\Delta:=\cH_\dR(E^\Delta_\univ)
\] on $Y_1^\Delta(\frn)_R$ to $X_1^\Delta(\frn)_R$. 
A similar extension is given by Gezm\.{i}\c{s}--Venkata \cite[Definition 7.5]{GV} in the case of full level.

Consider the exact sequence (\ref{EqnHodgeFil}) for the Tate-Drinfeld module $\TD^\td(\Lambda)|_{B^\Delta_R}$ over $\cP_R^\Delta=\Spec(B^\Delta_R)$. 
As in (\ref{EqnIdentificationHodgeFil}), we write
\[
M:=\DR(\TD^\td(\Lambda)|_{B^\Delta_R},\Ga)=B^\Delta_R\tau\oplus B^\Delta_R\tau^2.
\]
By (\ref{EqnEdX}), we have 
	\begin{equation}\label{EqnDfnDX}
	dX=a_1\tau+a_2\tau^2\in B^\Delta_R\tau\oplus B^\Delta_R\tau^2.
\end{equation}

On the other hand, by \cite[Lemma 4.1]{Ha_HTT}, we have
\[
	l(x):=\frac{da_1}{dx}-\frac{a_1}{a_2}\frac{da_2}{dx}\in \frac{1}{x}R[[x]]^\times.
\]
Following \cite[p.~247]{Gek_dR}, consider the element of $\DR(\TD(\Lambda)_{/R},\Ga)$ defined by
\[
\left(\nabla_{-x^2\frac{d}{dx}}+\frac{x^2}{a_2}\frac{da_2}{dx}\right)(dX),
\]
which is identified with 
\begin{equation}\label{EqnDfnEta}
\eta:=-x^2l(x)\tau\in R((x))\tau\oplus R((x))\tau^2.
\end{equation}
By the isomorphism (\ref{EqnIsomDeltaDeprive}), the natural map $R((x^{q-1}))\to B^\Delta_R$ factors through $R((x))$ and the Tate--Drinfeld module $
\TD^\td(\Lambda)|_{B^\Delta_R}$
agrees with the base extension of $\TD(\Lambda)_{/R}$ to $B^\Delta_R$. We denote the pull-back of $\eta$ via this map also by $\eta$.

Since $\sB_R^\Delta$ is defined as the normalization of $R[[x^{q-1}]]$ in $B_R^\Delta$, (\ref{EqnIsomDeltaDeprive}) also implies that the map $R[[x^{q-1}]]\to 
\sB_R^\Delta$ factors through $R[[x]]$.
Hence
\begin{equation}\label{EqnDfnLx}
l(x)\in \frac{1}{x}R[[x]]^\times\subseteq \frac{1}{x}(\sB_R^\Delta)^\times.
\end{equation}
This shows that the elements $dX$ and $\eta$ form a basis of the free $B^\Delta_R$-module $M=\DR(\TD^\td(\Lambda)|_{B^\Delta_R},
\Ga)$.

Now we define a $\sB^\Delta_R$-submodule $\sM$ of $M$ by 
\[
\sM:=\sB^\Delta_R dX\oplus \sB^\Delta_R \eta.
\]
For any $c\in \bF_q^\times$, consider the map $\theta_c:[c]^*M\to M$.
Since $a_i\in R[[x^{q-1}]]$ for $i=1,2$, we see that the element $x\frac{da_i}{dx}\in R((x))$ is $\bF_q^\times$-invariant.
Moreover, the element $\tau\in \Hom_{\bF_q,R((x))}(\TD(\Lambda)_{/R},\Ga)$ is given by $Z\mapsto X^q$ and, as an automorphism of $\TD(\Lambda)_{/R}$
over $R((x))$, the action of $c$ sends $\tau$ to $c\tau$.
Since (\ref{EqnIsomDeltaDeprive}) shows that the map $R((x))\to B^\Delta_R$ is $\bF_q^\times$-equivariant, we have
\[
\theta_c([c]^*dX)=cdX,\quad \theta_c([c]^*\eta)=\eta.
\]
This implies that $\sM$ is a modification of $\cH_{\dR,\univ}^\Delta$ at cusps. By Lemma \ref{LemExt}, we obtain a finite 
locally free sheaf $\bar{\cH}_{\dR,\univ}^\Delta$ 
on
$X_1^\Delta(\frn)_R$ extending $\cH_{\dR,\univ}^\Delta$.

\section{Autoduality and the Hodge bundle}\label{SecAD}

Let $(\frn,\Delta)$ be an $h$-level pair and let $R$ be an $A[1/\frn]$-algebra which is an excellent regular domain. 


\subsection{$F$-finiteness and formal \'{e}taleness around cusps}\label{SubsecFfinite}

Recall that an $\bF_p$-algebra $\cR$ is said to be $F$-finite if the $p$-th power Frobenius map $F:\cR\to \cR$ is finite. When $\cR$ is 
an $\bF_q$-algebra, it is the same as saying that the $q$-th power Frobenius map $\sigma:\cR\to \cR$ is finite. 
Examples of $F$-finite $\bF_p$-algebras include $A$, perfect rings and complete discrete valuation rings whose residue field has a finite $p$-basis. If $\cR$ is 
$F$-finite,
then any $\cR$-algebra of finite type and the localization of $\cR$ by any multiplicative subset are $F$-finite.

The following lemma can be proved as \cite[Lemma 110.44.4]{Stacks}.

\begin{lem}\label{LemFormallyEtale}
	Let $\cR$ be an $F$-finite excellent $\bF_p$-algebra and let $z$ be an indeterminate. Then the natural map
	\[
	\iota:\cR[z]\to \cR[[z]]
	\] 
	is formally \'{e}tale. In particular, we have an isomorphism of $\cR[[z]]$-modules
	\[
	\cR[[z]]dz=\cR[[z]]\otimes_{\cR[z]}\Omega^1_{\cR[z]/\cR}\to \Omega^1_{\cR[[z]]/\cR}.
	\]
	\end{lem}
	\begin{proof}
		It is enough to show that $\iota$ is formally unramified and formally smooth. For the former, note that
		the image $F(\cR)$ of the map $F:\cR\to \cR$ is a subring of $
		\cR$. By the assumption that $\cR$ is $F$-finite, \cite[Lemma 10.96.12]{Stacks} implies that $\cR[[z]]$ is also $F$-finite. Let $f_1,\ldots, f_r$ be 
		generators of the $F(\cR[[z]])$-module $\cR[[z]]$. Write any element $f\in \cR[[z]]$ as $f=\sum_{i=1}^r h_i^p f_i$ with some $h_i\in \cR[[z]]$.
		Then it follows that 
		\[
		df=\sum_{i=1}^r h_i^p df_i,
		\]
		which implies that $\Omega^1_{\cR[[z]]/\cR[z]}$ is a finite $\cR[[z]]$-module. On the other hand, for any integer $n\geq 0$, write
		\[
		f=P_n + z^n f_n\quad (P_n\in \cR[z],\ f_n\in \cR[[z]]).
		\]
		Then we have $df=z^n df_n$ in $\Omega^1_{\cR[[z]]/\cR[z]}$ for any $n$. Namely,
		\[
		df\in \bigcap_{n\geq 0}z^n \Omega^1_{\cR[[z]]/\cR[z]}.
		\]
		Since $\cR[[z]]$ is Noetherian, Krull's intersection theorem \cite[Lemma 10.51.5]{Stacks} implies $df=0$. Since $f$ is arbitrary, this means $
		\Omega^1_{\cR[[z]]/\cR[z]}=0$ and
		thus the map $\iota$ is formally unramified.
		
		Let us show the formal smoothness. Since $\cR$ is excellent, the ring $\cR[z]$ is also excellent and by \cite[Lemma 15.50.14]{Stacks}, the map $\iota$ is
		regular. By Popescu's theorem \cite[Theorem 16.12.1]{Stacks}, it is a filtered colimit of smooth maps.
		Since the formation of na\"{i}ve cotangent complexes is compatible with filtered colimit \cite[Lemma 10.134.9]{Stacks}, 
		the formal smoothness of $\iota$ follows from
		\cite[Proposition 10.138.8]{Stacks} and $\Omega^1_{\cR[[z]]/\cR[z]}=0$.
	\end{proof}
	
	\begin{cor}\label{CorBYFormallyEtale}
		Let $R$ be an $A[1/\frn]$-algebra which is an $F$-finite excellent regular domain. Then the natural morphisms
		\[
		\begin{gathered}
			\iota: \widehat{\Cusps}_R^\Delta\to X_1^\Delta(\frn)_R,\\
		\sT^\td:\cZ^\Delta_R\to X_1^\Delta(\frn)_R,\quad
		T^\td:[\Gamma_1^\Delta(\frn)]_{\TD^\td/R}\to 
		Y_1^\Delta(\frn)_R
		\end{gathered}
		\] 
		are formally \'{e}tale. In particular, we have isomorphisms
		\[
		\begin{gathered}
			\iota^*\Omega^1_{X_1^\Delta(\frn)_R/R}\to \Omega^1_{ \widehat{\Cusps}_R^\Delta/R},\\
			(\sT^\td)^*\Omega^1_{X_1^\Delta(\frn)_R/R}\to \Omega^1_{\sB^\Delta_R/R},\quad 
		(T^\td)^*\Omega^1_{Y_1^\Delta(\frn)_R/R}\to \Omega^1_{B^\Delta_R/R}.
		\end{gathered}
		\]
	\end{cor}
	\begin{proof}
		By Lemma \ref{LemFormallyEtale}, the map $R[1/j]\to R[[1/j]]$
		is formally \'{e}tale. Then the assertions on $\iota$ follows from the cartesian diagram (\ref{EqnDiagDfnCusps}) and \cite[Lemma 37.11.11]{Stacks}.
		Since the action of $\bF_q^\times$ on $\cZ_R^\Delta$ is free, the natural projection
		$\cZ_R^\Delta\to \cZ_R^\Delta/\bF_q^\times$
		is \'{e}tale. Hence the assertions on $\sT^\td$ follow from (\ref{EqnIsomCuspsQuotB}).
		Those on $T^\td$ follow by inverting $1/j$.
	\end{proof}

	\begin{rmk}\label{RmkFixHTT}
		Let $R_0$ be an $A[1/\frn]$-algebra which is an excellent regular domain. 
	As in \cite[\S6, p.~94]{Ha_DMC}, let $R_\frn$ be the domain generated over $R_0$ by a root of the $\frn$-th Carlitz 
		cyclotomic polynomial, which is a finite \'{e}tale $R_0$-algebra.
		By \cite[Lemma 6.2]{Ha_DMC}, the base change of $[\Gamma_1^\Delta(\frn)]_{\TD^\td/R_0}$ along the finite \'{e}tale map $R_0\to R_\frn$ is 
		a finite direct sum 
		of schemes of the form $\Spec(R_\frn((w)))$ with some indeterminate $w$.

		The proof of \cite[Corollary 4.2]{Ha_HTT} seems to have gaps. Firstly, in order that we have a natural map $Y(\frn)_R\to Y_1^\Delta(\frn)_R$,
		we need to assume that $C[\frn]|_R$ is isomorphic to $\underline{A/(\frn)}$. Thus, at the 
		beginning of the proof, 
		it is necessary to extend the base ring from $R_0$ to $R_\frn$. (The author thanks Paola Chilla for pointing out this issue.)
		
		Secondly,
		it is claimed in the last paragraph of \cite[p.~22]{Ha_HTT} that the pull-back of the map \cite[(4.4)]
		{Ha_HTT} to $R_\frn((w))$ is identified with a similar map for a Tate--Drinfeld module over $R_\frn((w))$. However, the target of the
		pull-back map is
		\[
		(R_\frn((w))\otimes_{B^\Delta_{R_\frn}} (T^\td)^*\Omega^1_{Y_1^\Delta(\frn)_{R_\frn}/R_\frn})^{\otimes q-1}
		\]
		and it is not necessarily isomorphic to $(\Omega^1_{R_\frn((w))/R_\frn})^{\otimes q-1}$. 
		
		When $R_0$ is $F$-finite, 
		the \'{e}taleness of the map $B^\Delta_{R_\frn}\to R_\frn((w))$ and Lemma \ref{LemFormallyEtale} show
		\[
		R_\frn((w))\otimes_{B^\Delta_{R\frn}} (T^\td)^*\Omega^1_{Y_1^\Delta(\frn)_{R_\frn}/R_\frn}\simeq R_\frn((w))\otimes_{B^\Delta_{R_\frn}} 
		\Omega^1_{B^\Delta_{R_\frn}/R_\frn}\simeq \Omega^1_{R_\frn((w))/
		R_\frn}
		\]
		and the issue is resolved. 
		
		By the compatibility with the base extension of Lemma \ref{LemExt}, the formation of the Hodge bundle $\bar{\omega}_\univ^\Delta$ is compatible 
		with any base change of excellent regular domains over $A[1/\frn]$ (see also \cite[Theorem 5.3]{Ha_DMC}).
		Since the ring $A[1/\frn]$ and finite $A[1/\frn]$-algebras are $F$-finite,
		the proof of \cite[Corollary 4.2]{Ha_HTT} 
		is valid with the following fix: we show it over $A[1/\frn]$ (by adding a root of the $\frn$-th Carlitz 
		cyclotomic polynomial for the first issue) and then extend the base ring to $R_0$.
		
		
	\end{rmk}
	
	\begin{rmk}\label{RmkFixDMC}
		The proof of \cite[Theorem 6.3 (3)]{Ha_DMC} needs a similar fix to Remark \ref{RmkFixHTT}.
		There it is allegedly proved that the pull-back of the sheaf $
		\Omega^1_{X_1^\Delta(\frn)_{R_0}/R_0}(2\Cusps_{R_0}^\Delta)$ by the map $\iota:\widehat{\Cusps}^\Delta_{R_0}\to 
		X_1^\Delta(\frn)_{R_0}$ is a free $\cO_{\widehat{\Cusps}^\Delta_{R_0}}$-module generated by $dx/x^2$. However, in the proof it is implicitly assumed that 
		\begin{itemize}
			\item $\iota^*\Omega^1_{X_1^\Delta(\frn)_{R_0}/R_0}\simeq \Omega^1_{\widehat{\Cusps}^\Delta_{R_0}/R_0}$, and
			\item $\Omega^1_{R_\frn[[w]]/R_\frn}=R_\frn[[w]]dw.$
		\end{itemize}
		Thus, in order to show \cite[Theorem 6.3 (3)]{Ha_DMC}, we need to assume that $R_0$ is $F$-finite, in which case the proof in 
		\cite{Ha_DMC} is valid since both of the assumptions above hold by Lemma \ref{LemFormallyEtale} and Corollary \ref{CorBYFormallyEtale}.

		This correction does not affect \cite{Ha_HTT}, since \cite[Theorem 6.3 (3)]{Ha_DMC} is used in that paper only when the base ring $R_0$ is $F$-finite.
		On the other hand, it does not seem that the assumption of $F$-finiteness of $R_0$ can be dropped by a base extension argument, since the formation of $
		\Omega^1_{\widehat{\Cusps}
		^\Delta_{R_0}/R_0}$ is not necessarily compatible with any base extension.

	\end{rmk}

\subsection{Autoduality of Drinfeld modules}\label{SubsecAD}

Let $h(z)$ be Gekeler's $h$-function \cite[Theorem 5.13]{Gek_Coeff}. It is a Drinfeld modular form of level $\mathit{GL}_2(A)$, weight $q+1$ and type one.
It is known that for the discriminant function $g_2(z)$ (which is denoted by $\Delta(z)$ in \cite{Gek_Coeff}), we have
	\begin{equation}\label{EqnH-functionRootOfDisc}
	h(z)^{q-1}=-g_2(z).
\end{equation} 
Moreover, \cite[Theorem 6.1]{Gek_Coeff} shows that, for a Carlitz period $
\bar{\pi}$ and the parameter at the infinity 
\[
u(z):=\frac{1}{\bar{\pi}z\prod_{a\in A\setminus\{0\}}\left(1-\frac{z}{a}\right)},
\]
the $h$-function has a $u$-expansion at the infinity of the form
\[
h(z)=u(z)+\sum_{n\geq 2} a_n u(z)^n,\quad a_n\in A.
\]

Note that the $\Cinf$-vector space of Drinfeld modular forms of level $\Gamma_1(\frn)$ and weight $k$ is identified with
\[
H^0(X_1^\Delta(\frn)_{\Cinf},(\bar{\omega}^\Delta_{\univ})^{\otimes k})
\]
(see \cite[p.~26]{Ha_HTT}).
The discriminant function $g_2(z)$ corresponds to an element of this space defined by the rule associating with a triple $(E,\lambda,\mu)$ 
consisting of a Drinfeld module $E$ of 
rank two over a scheme $S$ and a
$\Gamma_1^\Delta(\frn)$-structure $(\lambda,\mu)$ on it the section $\alpha_2^E(t)\in \omega_E^{q^2-1}(S)$.

Moreover, for any element $f$ of this space, let $f_\infty(u)$ be the $u$-expansion at the infinity of the corresponding Drinfeld modular form.
For the  $\infty$-cusp $x_\infty^\Delta:\Spec(\Cinf[[x]])\to X_1^\Delta(\frn)_{\Cinf}$ \cite[p.~93]{Ha_DMC}, write 
\[
F_\infty(x) (dX)^{\otimes k}=(x_{\infty}^\Delta)^*f
\]
with some $F_\infty(x)\in \Cinf[[x]]$.
Then we have $f_\infty(u)=F_\infty(u)$.
Therefore, by the $x$-expansion principle \cite[Proposition 4.8 (i)]{Ha_HTT} and (\ref{EqnH-functionRootOfDisc}), the $h$-function defines a section 
\[
\mathfrak{h}\in H^0(X_1^\Delta(\frn)_{A[1/\frn]},(\bar{\omega}^\Delta_{\univ})^{\otimes q+1})
\]
satisfying
\begin{equation}\label{EqnHsectionRelation}
	\mathfrak{h}^{\otimes q-1}=-\alpha_2^{E^\Delta_\univ}(t).
\end{equation}
In particular, $\mathfrak{h}$ is nowhere vanishing on $Y_1^\Delta(\frn)_{A[1/\frn]}$.

As mentioned in \cite[Remark 2.20]{Ha_HTT}, this gives an autoduality of Drinfeld modules of rank two admitting a $\Gamma_1^\Delta(\frn)$-structure. 

\begin{thm}\label{ThmAutodual}
	Let $S$ be a scheme over $A[1/\frn]$. Let $E=(\cL_E,\Phi^E)$ be a Drinfeld module of rank two over $S$ with a $\Gamma_1^\Delta(\frn)$-structure
	$(\lambda,\mu)$. Then there exists an isomorphism 
	\[
	\AD_{(E,\lambda,\mu)}: E\to E^D
	\]
	of Drinfeld modules over $S$. 
\end{thm}
\begin{proof}
	The isomorphism class of $(E,\lambda,\mu)$ defines a morphism $f:S\to Y_1^\Delta(\frn)_{A[1/\frn]}$. By the relation (\ref{EqnHsectionRelation}),
	the element $f^*\mathfrak{h}\in \cL_E^{\otimes -1-q}(S)$ is nowhere 
	vanishing on $S$ and satisfies
	\begin{equation}\label{EqnH-functionSectionRelation}
	(f^*\mathfrak{h})^{\otimes q-1}=-\alpha_2^E(t)\in \cL_E^{\otimes 1-q^2}(S).
	\end{equation}
	Then the isomorphism of $\cO_S$-modules
	\[
	\cL_E\to \cL_E^{\otimes -q},\quad l\mapsto l\otimes f^*\mathfrak{h}
	\]
	gives an isomorphism of $\bF_q$-module schemes 
	\[
	\AD_{(E,\lambda,\mu)}:E=\bV_*(\cL_E)\to \bV_*(\cL_E^{\otimes -q})=E^D. 
	\]
	By (\ref{EqnTmodStrOnDualDM}), we see that it 
	is $A$-linear. 
	This concludes the proof.
\end{proof}

By Theorem \ref{ThmAutodual}, we also have an isomorphism of $\cO_S$-modules
\begin{equation}\label{EqnAutodualHodge}
\AD_{(E,\lambda,\mu)}^*:\omega_{E^D}=\cL_E^{\otimes q}\to \cL_E^{\otimes -1}=\omega_E ,\quad l\mapsto l\otimes f^*\mathfrak{h}.
\end{equation}

\begin{lem}\label{LemH-FunctionOnB}
	Let $b\in B_R^\Delta$ be any element satisfying $b^{q-1}=-a_2$. Then we have $b\in x(\sB_R^\Delta)^\times$.
\end{lem}
\begin{proof}
	In the $R((x^{q-1}))$-algebra $B_R^\Delta$, (\ref{EqnTDCoeff}) implies
	\[
	\left(x^{-1}b\right)^{q-1}=-x^{1-q}a_2\in R[[x^{q-1}]]^\times.
	\]
	Since $\sB^\Delta_R$ is defined as the normalization of $R[[x^{q-1}]]$ in $B^\Delta_R$, we have $x^{-1}b\in (\sB^\Delta_R)^\times$.
\end{proof}

\begin{cor}\label{CorH-FunctionOnP}
	For the map $T^\td:\cP^\Delta_R=\Spec(B^\Delta_R)\to Y_1^\Delta(\frn)_R$, write
	\[
	(T^\td)^*\mathfrak{h}=b_h (dX)^{\otimes q+1},\quad b_h\in B^\Delta_R.
	\]
	Then we have $b_h^{q-1}=-a_2$ and $b_h\in x(\sB^\Delta_R)^\times$.
\end{cor}
\begin{proof}
	By (\ref{EqnEunivPullBackP}) and the relation (\ref{EqnH-functionSectionRelation}), we have $b_h^{q-1}=-a_2$. Then 
	Lemma \ref{LemH-FunctionOnB} concludes the proof.
\end{proof}

\subsection{Relationship between autoduality and $\Gamma_1^\Delta(\frn)$-structures}\label{SubsecIntrinsic}

Here we give an intrinsic construction of the Drinfeld modular curve $Y_1^\Delta(\frn)_R$ which leads to a more direct proof of Theorem \ref{ThmAutodual}.
The content of this subsection will not be used in the rest of the paper.
We denote the fiber product and the tensor product over $\bF_q$ by $\times$ and $\otimes$, respectively.
We use the notation and definitions from \cite[\S7]{Boeckle}.

\begin{dfn}\label{DfnH-structure}
	Let $S$ be a scheme over $A$ and let $E=(\cL_E,\Phi^E)$ be a Drinfeld module of rank two over $S$.
	An $h$-structure on $E$ is a section $H\in \cL_E^{\otimes -1-q}(S)$ satisfying 
	\[
	H^{\otimes q-1}=-\alpha_2^E(t).
	\] 
	If an $h$-structure $H$ on $E$ exists, then the section $H$ is nowhere vanishing and the invertible $\cO_S$-module $\cL_E^{\otimes -1-q}$ is 
	trivial.
\end{dfn}

Let $E_\univ=(\cL_\univ,\Phi^{E_\univ})$ be the universal Drinfeld module over $Y_1(\frn)_R$. Then the element 
$\alpha_2^{E_\univ}(t)\in \cL_\univ^{\otimes 1-q^2}(Y_1(\frn))$ defines a section
\[
\alpha_2:Y_1(\frn)_R\to \bV_*(\cL_\univ^{\otimes 1-q^2}).  
\]
Consider the morphism of schemes over $Y_1(\frn)_R$
\[
\bV_*(\cL_\univ^{\otimes -1-q})\to \bV_*(\cL_\univ^{\otimes 1-q^2}),\quad l\mapsto l^{q-1}.
\]
Define a scheme $Y_1^h(\frn)_R$ by the cartesian diagram
\[
\xymatrix{
	Y_1^h(\frn)_R\ar[r] \ar[d]& Y_1(\frn)_R \ar[d]^{\alpha_2}\\
	\bV_*(\cL_\univ^{\otimes -1-q})\ar[r] & \bV_*(\cL_\univ^{\otimes 1-q^2}).
}
\]
Then $Y_1^h(\frn)_R$ represents the functor over $Y_1(\frn)_R$ defined by
\[
S\mapsto \{H\in (\cL_\univ^{\otimes -1-q}|_S)(S)\mid H^{\otimes q-1}=-\alpha_2^{E_\univ}(t)|_S\}.
\]
The group $\bF_q^\times$ acts on $Y_1^h(\frn)_R$ over $Y_1(\frn)_R$ by
\[
[c](H)=c^{-1}H\quad (c\in \bF_q^\times).
\]
This makes $Y_1^h(\frn)_R$ a finite \'{e}tale $\bF_q^\times$-torsor over $Y_1(\frn)_R$.

For any scheme $S$ over $A$ and any Drinfeld module $E$ of rank $r$ over $S$, the $\cO_S$-module
\[
\sHom_{\bF_q,S}(E,\Ga)=\bigoplus_{m\geq 0} \cL_E^{\otimes -q^m}
\]
is endowed with the right $A$-action via $\Phi^E$. We denote by $\cM(E)$ the associated $\tau$-sheaf over $A$ on $S$
\cite[Example 7.7 (b)]{Boeckle}. It is a locally free
$\cO_{S\times\Spec(A)}$-module of rank $r$ such that, for the first projection $\prjt_1:S\times \Spec(A)\to S$, we have
\[
(\prjt_1)_*\cM(E)=\sHom_{\bF_q,S}(E,\Ga).
\]
Moreover, it is equipped with an $\cO_{S\times\Spec(A)}$-linear map $\tau:(\sigma_S\times 1)^*\cM(E)\to \cM(E)$ induced by
\[
\sHom_{\bF_q,S}(E,\Ga)^{(q)}\to \sHom_{\bF_q,S}(E,\Ga),\quad 1\otimes f\mapsto \tau\circ f.
\]
We denote by
\[
\cD(E):=\bigwedge^r\cM(E)
\]
the exterior power of $\cM(E)$ as an $\cO_{S\times\Spec(A)}$-module. We consider $\cD(E)$ as a $\tau$-sheaf with the map $\bigwedge^r\tau$.

\begin{lem}\label{LemDetE}
	The $\tau$-sheaf $\cD(E)$ is isomorphic to $\prjt_1^*\cL_E^{\otimes -\frac{1-q^r}{1-q}}$ with the $\tau$-structure
	\[
	(-1)^r(\theta-t)\prjt_1^*(\alpha_r^E(t)^{\otimes -1}\tau).
	\]
\end{lem}
\begin{proof}
	Let $U\subseteq S$ be an affine open subscheme such that $\cL_E|_U$ is trivial. Write $\cL_E|U=\cO_U l$ with some nowhere vanishing section $l$.
	Then $\cM(E)|_{U\times \Spec(A)}$ is a free $\cO_{U\times \Spec(A)}$-module generated by 
	\[
	l^{\otimes -1},l^{\otimes -q},\ldots, l^{\otimes -q^{r-1}}.
	\]
	Hence we have
	\[
	\cD(E)|_{U\times \Spec(A)}=\cO_{U\times \Spec(A)} l^{\otimes -\frac{1-q^r}{1-q}}
	\]
	and, if we write $\alpha_r^E(t)=g_r l^{\otimes 1-q^r}$ with some $g_r\in \cO(U)^\times$, then its $\tau$-structure is given by
	\[
	(\sigma_U\times 1)^*l^{\otimes -\frac{1-q^r}{1-q}}\mapsto (-1)^r(\theta-t)\prjt_1^*(g_r^{-1}) l^{\otimes -\frac{1-q^r}{1-q}}.
	\]
	Then the trivialization
	\[
	\left(\prjt_1^*\cL_E^{\otimes -\frac{1-q^r}{1-q}}\middle)\right|_{U\times\Spec(A)}=\cO_{U\times\Spec(A)}\prjt_1^*l^{\otimes -\frac{1-q^r}{1-q}}
	\]
	gives an isomorphism
	\[
	\begin{aligned}
	\left(\prjt_1^*\cL_E^{\otimes -\frac{1-q^r}{1-q}}\middle)\right|_{U\times\Spec(A)}&\to  \cD(E)|_{U\times\Spec(A)},\\
	\prjt_1^*l^{\otimes -\frac{1-q^r}{1-q}}&\mapsto l^{\otimes -\frac{1-q^r}{1-q}},
	\end{aligned}
	\]
	which glues to yield an isomorphism of $\tau$-sheaves as claimed.
\end{proof}


For the Carlitz module $C$, its associated $\tau$-sheaf is given by
\[
\cM(C)=\cO_{\Spec(A)\times \Spec(A)},\quad \tau(1)=t-\theta.
\]

\begin{lem}\label{LemDetEHCarlitz}
	Let $S$ be a scheme over $A$ and let $E$ be a Drinfeld module of rank two over $S$ equipped with an $h$-structure $H$.
	Then we have an isomorphism of $\tau$-sheaves over $A$ on $S$
	\[
	\nu_H: \cD(E)\to \cM(C|_S)
	\]
	satisfying $\nu_{[c](H)}=c\nu_H$ for any $c\in \bF_q^\times$.
\end{lem}
\begin{proof}
	Multiplying $H^{\otimes -1}$ induces an isomorphism of $\cO_{S\times \Spec(A)}$-modules
	\[
	\nu_H:\prjt_1^* \cL_E^{\otimes -1-q}\to \prjt_1^* \cO_S=\cO_{S\times \Spec(A)}=\cM(C|_S).
	\]
	Lemma \ref{LemDetE} shows that it is an isomorphism of $\tau$-sheaves. Since $\nu_H(\prjt_1^*H)=1$ and $\nu_{[c](H)}(c^{-1}\prjt_1^*H)=1$,
	we obtain $\nu_{[c](H)}=c\nu_H$.
\end{proof}

The following lemma can be shown in the same way as \cite[Proposition 7.19]{Boeckle}.

\begin{lem}\label{LemBoeckle}
	Let $B$ be a Noetherian normal $A[1/\frn]$-algebra and let $E=(\cL_E,\Phi^E)$ be a Drinfeld module of some rank over $S=\Spec(B)$.
	Then we have a natural isomorphism 
	\[
	E[\frn]\to \sHom_A((\cM(E)/\frn\cM(E))_{\mathrm{\acute{e}t}},\underline{A/(\frn)})
	\]
	of \'{e}tale sheaves of $A/(\frn)$-modules on $S$, where $(-)_{\mathrm{\acute{e}t}}$ is the $A$-linear 
	functor attaching an \'{e}tale sheaf to a $\tau$-sheaf 
	\cite[Definition 7.16]{Boeckle}.
\end{lem}
\begin{proof}
	For any \'{e}tale ring map $u:B\to B'$, we have a natural map
	\[
	\begin{aligned}
		E[\frn](B')&\to \Hom_{\tau,A,B'}(u^*\cM(E)/\frn u^*\cM(E),\Hom_{\bF_q}(A/(\frn),B')),\\
		e&\mapsto (f\mapsto (a\mapsto (f(ae)))).
	\end{aligned}
	\]
	Here $\Hom_{\tau,A,B'}$ denotes the group of $(B'\otimes A)$-linear homomorphisms compatible with $\tau$
	and the $\tau$-structure on the ring $B'$ is given by the $q$-th power Frobenius map.
	
	Note that we have isomorphisms
	of $(B'\otimes A)$-modules
	\[
	A/(\frn)\otimes B'\to \Hom_{\bF_q}(A/(\frn),\bF_q)\otimes B'\to \Hom_{\bF_q}(A/(\frn),B'),
	\]
	where the left one is induced by the perfect pairing
	\[
	A/(\frn)\otimes A/(\frn)\to \bF_q,\quad (a\bmod \frn,b\bmod \frn)\mapsto \Res_\infty(\frn^{-1}abdt).
	\]
	Since $B$ is Noetherian and normal, so is $B'$ and we can write
	\[
	B'=\prod_{i\in I}B'_i,\quad |I|<+\infty
	\]
	with some normal domain $B'_i$. This implies that the $\tau$-invariant subring of $B'$ equals
	\[
	(B')^\tau=\prod_{i\in I} \bF_q
	\]
	and we have $(A/(\frn)\otimes B')^{\tau}=\underline{A/(\frn)}(B')$.
	Thus we obtain a map of \'{e}tale sheaves
	\[
	E[\frn]\to \sHom_A((\cM(E)/\frn\cM(E))_{\mathrm{\acute{e}t}},\underline{A/(\frn)}).
	\]
	Since \cite[Proposition 1.8.3]{Anderson} shows that it is isomorphic at any geometric points, the lemma follows. 
\end{proof}

The following proposition gives a Weil pairing valued in $C[\frn]$ for a Drinfeld module $E$ of rank two admitting an $h$-structure.
Over $\Cinf$, the observation that we can define such a Weil pairing using Gekeler's $h$-function as an $h$-structure also appeared in 
\cite[\S4]{Breuer}.

\begin{prop}\label{PropWeilPair}
	Let $B$ be a Noetherian normal $A[1/\frn]$-algebra and let $E=(\cL_E,\Phi^E)$ be a Drinfeld module of rank two over $S=\Spec(B)$.
	Let $H$ be an $h$-structure on $E$. Then we have an isomorphism
	\[
	f_H: \bigwedge^2 E[\frn]\to  C[\frn]|_B
	\]
	of \'{e}tale sheaves of $A/(\frn)$-modules on $S$. For any $c\in \bF_q^\times$, it satisfies
	 \begin{equation}\label{EqnWeilPairHEquivar}
	 f_{[c]H}=c^{-1}f_H.
	 \end{equation}
\end{prop}
\begin{proof}
	Applying Lemma \ref{LemBoeckle} to the Carlitz module $C$, we have an isomorphism
	\[
	\phi_C:C[\frn]|_B\to \sHom_A((\cM(C|_B)/\frn\cM(C|_B))_{\mathrm{\acute{e}t}},\underline{A/(\frn)})
	\]
	of \'{e}tale sheaves of $A/(\frn)$-modules on $S$.
	
	On the other hand, since the functor $(-)_{\mathrm{\acute{e}t}}$ commutes with the exterior power over $A/(\frn)$ \cite[Theorem 7.18]{Boeckle}
	and the formation of exterior power commutes with taking the linear dual, the 
	isomorphism of 
	Lemma \ref{LemBoeckle}
	induces an isomorphism 
	\[
	\delta_E:\bigwedge^2 E[\frn]\to \sHom_A((\cD(E)/\frn\cD(E))_{\mathrm{\acute{e}t}},\underline{A/(\frn)})
	\]
	of \'{e}tale sheaves of $A/(\frn)$-modules on $S$. Composing these isomorphisms with the induced map by $\nu_H$ of Lemma \ref{LemDetEHCarlitz}, we obtain an 
	isomorphism
	\[
	\begin{aligned}
	f_H=\phi_C^{-1}\circ(\nu_H^*)^{-1}\circ\delta_E: \bigwedge^2 E[\frn]&\to \sHom_A((\cD(E)/\frn\cD(E))_{\mathrm{\acute{e}t}},\underline{A/(\frn)})\\
	&\to \sHom_A((\cM(C|_B)/\frn\cM(C|_B))_{\mathrm{\acute{e}t}},\underline{A/(\frn)})\\
	&\to C[\frn]|_B
	\end{aligned}
	\]
	of \'{e}tale sheaves of $A/(\frn)$-modules on $S$.

	Moreover, for any $c\in \bF_q^\times$, Lemma \ref{LemDetEHCarlitz} also implies
	\[
	f_{[c](H)}=\phi_C^{-1}\circ(\nu_{[c](H)}^*)^{-1}\circ\delta_E=\phi_C^{-1}\circ(c\nu_H^*)^{-1}\circ\delta_E=c^{-1}f_H.
	\]
	This concludes the proof of the proposition.
\end{proof}

\begin{cor}\label{CorHstrMu}
	Let $B$ be a Noetherian normal $A[1/\frn]$-algebra and let $E=(\cL_E,\Phi^E)$ be a Drinfeld module of rank two over $S=\Spec(B)$ equipped with a $
	\Gamma_1(\frn)$-structure $\lambda$.
	Let $H$ be an $h$-structure on $E$. Then we have an isomorphism
	\[
	\mu_H: \underline{A/(\frn)}\to E[\frn]/\Img(\lambda)
	\]
	of $A$-module schemes over $S$. For any $c\in \bF_q^\times$, it satisfies
	\begin{equation}\label{EqnMuHEquivar}
		\mu_{[c]H}=c\mu_H.
	\end{equation}
\end{cor}
\begin{proof}
	The $\Gamma_1(\frn)$-structure $\lambda:C[\frn]|_B\to E[\frn]$ induces an isomorphism 
	\[
	g:\bigwedge^2 E[\frn]\simeq C[\frn]|_B\otimes_{A/(\frn)} (E[\frn]/\Img(\lambda))
	\]
	of \'{e}tale sheaves of $A/(\frn)$-modules on $S$. With the isomorphism $f_H$ of Proposition \ref{PropWeilPair},
	this gives an isomorphism
	\[
	g\circ f_H^{-1}: C[\frn]|_B\to C[\frn]|_B\otimes_{A/(\frn)} (E[\frn]/\Img(\lambda)).
	\]
	
	Note that we have an isomorphism
	\[
	C[\frn]|_B\otimes_{A/(\frn)} \sHom_{A}(C[\frn]|_B,\underline{A/(\frn)})\to \underline{A/(\frn)}
	\]
	of \'{e}tale sheaves of $A/(\frn)$-modules over $S$.
	Then, twisting by the \'{e}tale sheaf $\sHom_{A}(C[\frn]|_B,\underline{A/(\frn)})$,
	we obtain an isomorphism
	\[
	\mu_H=(g\circ f_H^{-1})(-1): \underline{A/(\frn)}\to E[\frn]/\Img(\lambda).
	\]
	
	Moreover, for any $c\in \bF_q^\times$, (\ref{EqnWeilPairHEquivar}) yields
	\[
	\mu_{[c](H)}=(g\circ f_{[c](H)}^{-1})(-1)=(g\circ (c^{-1}f_H)^{-1})(-1)=c\mu_H.
	\]
	This concludes the proof.
\end{proof}

\begin{thm}\label{ThmHstrVSGammaDeltaStr}
	We have an isomorphism 
	\[
	j^\Delta:Y_1^h(\frn)_R\to Y_1^\Delta(\frn)_R
	\]
	of schemes over $Y_1(\frn)_R$.
\end{thm}
\begin{proof}
	Since the affine ring of $Y_1(\frn)_R$ is Noetherian and normal, so is that of $Y_1^h(\frn)_R$. For the universal object $(E_\univ, \lambda_\univ)$
	over $Y_1(\frn)_R$ and its universal $h$-structure $H_\univ$ over $Y_1^h(\frn)_R$, Corollary \ref{CorHstrMu} gives an isomorphism
	\[
	\mu_{H_\univ}: \underline{A/(\frn)}\to (E_\univ[\frn]/\Img(\lambda_\univ))|_{Y_1^h(\frn)_R}
	\] 
	of $A$-module schemes over $Y_1^h(\frn)_R$. Then the image of 
	\[
	\mu_{H_\univ}\in I_{(E_\univ,\lambda_\univ)}(Y^h_1(\frn)_R)\to (I_{(E_\univ,\lambda_\univ)}/\Delta)(Y^h_1(\frn)_R)
	\] 
	yields a morphism $j^\Delta: Y_1^h(\frn)_R\to Y_1^\Delta(\frn)_R$ over 
	$Y_1(\frn)_R$.

	The source and the target of $j^\Delta$ are finite \'{e}tale $\bF_q^\times$-torsors over $Y_1(\frn)_R$,
	and (\ref{EqnMuHEquivar}) implies that $j^\Delta$ is $\bF_q^\times$-equivariant. Hence $j^\Delta$ is an isomorphism.
\end{proof}

\begin{cor}\label{CorExtH}
	Let $H$ be an $h$-structure on $E_\univ^\Delta$, which exists by Theorem \ref{ThmHstrVSGammaDeltaStr}. Then $H$ extends to an element of 
	 \[
	 H^0(X_1^\Delta(\frn)_R,(\bar{\omega}_\univ^\Delta)^{\otimes q+1}).
	 \]
\end{cor}
\begin{proof}
	By Lemma \ref{LemExt} and \cite[(15.91.16.1)]{Stacks}, it is enough to show
	\[
	(T^\td)^*H\in \sB_R^\Delta (dX)^{\otimes q+1}
	\]
	under the isomorphism (\ref{EqnEunivPullBackP}).
	Write $(T^\td)^*H=b (dX)^{\otimes q+1}$ with some $b\in B_R^\Delta$.
	Since $H$ is an $h$-structure on $E_\univ^\Delta$, by (\ref{EqnEunivPullBackP}) we have $b^{q-1}=-a_2$. Then Lemma \ref{LemH-FunctionOnB}
	yields $b\in \sB_R^\Delta$ and the corollary follows.
\end{proof}

Using Corollary \ref{CorExtH}, we can bypass the use of the $x$-expansion principle in the proof of Theorem \ref{ThmAutodual}.

\subsection{Extension of the Hodge filtration and autoduality}\label{SubsecExtHodgeFil}

By Theorem \ref{ThmAutodual}, the universal object $(E_\univ^\Delta,\lambda_\univ,\mu_\univ)$ over $Y_1^\Delta(\frn)_R$ gives an isomorphism
\[
\AD_{E^\Delta_\univ}: E_\univ^\Delta\to (E_\univ^\Delta)^D.
\]
Moreover, (\ref{EqnHodgeFil}) for $E_\univ^\Delta$ yields an exact sequence
\begin{equation}\label{EqnHodgeFilNaiveUniv}
	\xymatrix{
		0 \ar[r] & \omega_\univ^\Delta \ar[r] & \cH_{\dR,\univ}^\Delta \ar[r] & \Lie((E_\univ^\Delta)^D) \ar[r] &0.
	}
\end{equation}
By composing the right map of (\ref{EqnHodgeFilNaiveUniv}) with the dual of the map $\AD_{E^\Delta_\univ}^*$ of (\ref{EqnAutodualHodge}),
we obtain 
a morphism of $\cO_{Y_1^\Delta(\frn)_R}$-modules
\[
\prjt: \cH^\Delta_{\dR,\univ}\to \Lie((E^\Delta_\univ)^D)=(\omega_{(E_{\univ}^\Delta)^D})^\vee\simeq \omega_{E_{\univ}^\Delta}^\vee
\]
and an exact sequence 
\begin{equation}\label{EqnHodgeFilUniv}
\xymatrix{
	0 \ar[r] & \omega_\univ^\Delta \ar[r] & \cH_{\dR,\univ}^\Delta \ar[r]^{\prjt} & (\omega_{\univ}^\Delta)^\vee \ar[r] &0
}
\end{equation}
of finite locally free sheaves on $Y_1^\Delta(\frn)_R$.

\begin{thm}\label{ThmExtProjToLie}
	The map $\prjt$ extends to a surjection of $\cO_{X_1^\Delta(\frn)_R}$-modules
	\[
	\bar{\prjt}: \bar{\cH}^\Delta_{\dR,\univ}\to (\bar{\omega}_{E_{\univ}^\Delta})^\vee
	\]
	which sits in an exact sequence
	\[
	\xymatrix{
		0 \ar[r] & \bar{\omega}_\univ^\Delta \ar[r] & \bar{\cH}_{\dR,\univ}^\Delta \ar[r]^{\bar{\prjt}} & (\bar{\omega}_{\univ}^\Delta)^\vee \ar[r] &0
	}
	\]
	of finite locally free sheaves on $X_1^\Delta(\frn)_R$ extending (\ref{EqnHodgeFilUniv}).
\end{thm}
\begin{proof}
	By Lemma \ref{LemExt}, it is enough to show that the map $(T^\td)^*(\prjt)$ induces a surjection $(\sT^\td)^*\bar{\cH}_{\dR,\univ}^\Delta\to 
	(\sT^\td)^*((\bar{\omega}_{\univ}^\Delta)^\vee)$ whose kernel is $\sB^\Delta_R dX$.
	Since the formation of the linear dual of a finite locally free sheaf is compatible with any base change, for the $\cO_{X_1^\Delta(\frn)_R}$-module $
	\bar{\cF}$ corresponding to an object $(\cF,\sM)$ of the category $\cC_R$ of Lemma \ref{LemExt}, we see that $(\sT^\td)^*(\bar{\cF}^\vee)$ is the 
	coherent sheaf corresponding to the $\sB^\Delta_R$-module
	\[
	\Hom_{\sB^\Delta_R}(\sM,\sB^\Delta_R).
	\]
	
	Let us describe the map $(T^\td)^*(\prjt)$.
	Since the formation of the dual Drinfeld module is compatible with base extension, (\ref{EqnEunivPullBackP}) implies that 
	the pull-back of $(E^\Delta_\univ)^D$ by the map
	$T^\td$ is $\TD^\td(\Lambda)^D|_{B^\Delta_R}$. By (\ref{EqnIdentificationXY}), for $Y=X^{\otimes -q}$ we have
	\[
	\TD^\td(\Lambda)^D=\Spec(R((x^{q-1}))[Y]),\quad \Phi^{\TD^\td(\Lambda)^D}_t(Y)=\theta Y-a_1a_2^{-1} Y^q+a_2^{-q}Y^{q^2}.
	\]
	
	
	Note that the pull-back of the exact sequence (\ref{EqnHodgeFilNaiveUniv}) by $T^\td$ is
	\begin{equation}\label{EqnExactPropExtHodgeFil}
		\xymatrix{
			0 \ar[r] & B^\Delta_R dX \ar[r] &  B^\Delta_R dX\oplus B^\Delta_R\eta \ar[r] & B^\Delta_R\frac{d}{dY}\ar[r] & 0,
		}
	\end{equation}
	where the element $\eta$ is defined by (\ref{EqnDfnEta}).
	The left map of (\ref{EqnExactPropExtHodgeFil}) is the natural inclusion. By (\ref{EqnELie}), the right map equals
	\[
	b_1 \tau+b_2\tau^2\mapsto \frac{a_2b_1-a_1b_2}{a_2}\frac{d}{dY}\in \Lie(\TD^\td(\Lambda)^D|_{B^\Delta_R}).
	\]
	By Corollary \ref{CorH-FunctionOnP}, for an element $b_h\in x(\sB_R^\Delta)^\times\subseteq (B_R^\Delta)^\times$, we have
	\[
	(T^\td)^*(\prjt): B^\Delta_R\tau\oplus B^\Delta_R\tau^2 \to B^\Delta_R\frac{d}{dX},\quad b_1\tau+b_2\tau^2\mapsto \frac{a_2b_1-a_1b_2}{a_2b_h}\frac{d}{dX}. 
	\]
	Then (\ref{EqnDfnLx}) implies
	\[
	(T^\td)^*(\prjt)(\eta)=-\frac{x^2l(x)}{b_h}\frac{d}{dX}\in (\sB^\Delta_R)^\times \frac{d}{dX}.
	\]
	Hence (\ref{EqnExactPropExtHodgeFil}) induces an exact sequence of $\sB^\Delta_R$-modules
	\[
	\xymatrix{
		0 \ar[r] & \sB^\Delta_R dX \ar[r] &  \sB^\Delta_R dX\oplus \sB^\Delta_R\eta \ar[r] & \sB^\Delta_R\frac{d}{dX}\ar[r] & 0.
	}
	\]
	This concludes the proof of the theorem.
\end{proof}

\subsection{Kodaira--Spencer isomorphism and autoduality}\label{SubsecExtKS}

By the proof of \cite[Corollary 4.2]{Ha_HTT} and Remark \ref{RmkFixHTT}, we know that the map 
\[
\KS^\vee_{E^\Delta_\univ/Y_1^\Delta(\frn)_R/R}:\omega_{E^\Delta_\univ}\otimes_{\cO_{Y_1^\Delta(\frn)_R}} \omega_{(E^\Delta_\univ)^D}\to 
\Omega^1_{Y_1^\Delta(\frn)_R/R}
\]
is an isomorphism. By composing it with the map (\ref{EqnAutodualHodge}) for the universal object $(E_\univ^\Delta,\lambda_\univ,\mu_\univ)$, we obtain an 
isomorphism
\[
\KS^\curlyvee: (\omega_\univ^\Delta)^{\otimes 2}\to \Omega_{Y_1^\Delta(\frn)_R/R}^1
\] 
of invertible sheaves on $Y_1^\Delta(\frn)_R$

\begin{thm}\label{ThmKS}
	The map $\KS^\curlyvee$ extends to an isomorphism
	\[
	\bar{\KS}^\curlyvee:  (\bar{\omega}_\univ^\Delta)^{\otimes 2}\to \Omega_{X_1^\Delta(\frn)_R/R}^1(2\Cusps_R^\Delta)
	\]
	of invertible sheaves on $X_1^\Delta(\frn)_R$.
\end{thm}
\begin{proof}
	By Lemma \ref{LemBCCusps}, (\ref{EqnDiagKSBCR}) and the compatibility with base extension of Lemma \ref{LemExt}, it is enough to show the theorem for $R=A[1/
	\frn]$.
	Thus we may assume that $R$ is $F$-finite.
	
	By Lemma \ref{LemExt}, it is enough to show that the map $(T^\td)^*\KS^\curlyvee$ induces an isomorphism 
	\[
	(\sT^\td)^*((\bar{\omega}_\univ^\Delta)^{\otimes 2})\to (\sT^\td)^*(\Omega_{X_1^\Delta(\frn)_R/R}^1(2\Cusps_R^\Delta)).
	\]
	By Corollary \ref{CorBYFormallyEtale} and (\ref{EqnDiagKSBCB}), 
	the map $(T^\td)^*\KS^\curlyvee$ is identified with a similar map 
	\begin{equation}\label{EqnKScvTDB}
	\KS^\curlyvee_{\TD^\td|_{B^\Delta_R}}:  (B^\Delta_R dX)^{\otimes 2}\to B^\Delta_R dX \otimes_{B^\Delta_R} B^\Delta_R dY\overset{\KS^\vee}{\to} 
	\Omega^1_{B^\Delta_R/R}
	\end{equation}
	for $\TD^\td(\Lambda)|_{B^\Delta_R}$, where $\KS^\vee=\KS^\vee_{\TD^\td(\Lambda)|_{B_R^\Delta}/B_R^\Delta/R}$ and the left map is given by 
	\[
	(dX)^{\otimes 2}\mapsto b_h^{-1}dX\otimes dY
	\]
	with the element $b_h$ of Corollary \ref{CorH-FunctionOnP}.
	By (\ref{EqnIsomDeltaDeprive}), the map $R((x^{q-1}))\to B^\Delta_R$ factors through an \'{e}tale map $R((x))\to B^\Delta_R$. 
	Thus, by (\ref{EqnDiagKSBCB}), the right map of (\ref{EqnKScvTDB}) agrees with the pull-back to $B^\Delta_R$ of the map
	\[
	\KS^\vee_{\TD(\Lambda)_{/R}/R((x))/R}:  R((x))dX\otimes_{R((x))} R((x))dY\to \Omega^1_{R((x))/R}
	\]
	for $\TD(\Lambda)_{/R}$. Then \cite[Lemma 4.1]{Ha_HTT} yields
	\[
	\KS^\curlyvee_{\TD^\td|_{B^\Delta_R}}((dX)^{\otimes 2})=b_h^{-1}l(x)dx.
	\]
	
	By Corollary \ref{CorBYFormallyEtale}, we have an isomorphism
	\[
	\Omega_{X_1^\Delta(\frn)_R/R}^1(2\Cusps_R^\Delta)|_{\widehat{\Cusps}_R^\Delta}\to \Omega^1_{\widehat{\Cusps}_R^\Delta/R}(2\Cusps_R^\Delta).
	\]
	Then \cite[Theorem 6.3 (3)]{Ha_DMC} and Remark \ref{RmkFixDMC} show that this is a free $\cO(\widehat{\Cusps}_R^\Delta)$-module of rank one
	generated by $x^{-2}dx$. 
	Thus we have an isomorphism of $\sB^\Delta_R$-modules
	\[
	(\sT^\td)^*(\Omega^1_{X_1^\Delta(\frn)_R/R}(2\Cusps_R^\Delta))\to x^{-2}\sB^\Delta_R dx.
	\]

	Since (\ref{EqnDfnLx}) and Corollary \ref{CorH-FunctionOnP} imply $b_h^{-1}l(x)\in x^{-2}(\sB^\Delta_R)^\times$, the map $\KS^\curlyvee_{\TD^\td|
	_{B^\Delta_R}}$ induces 
	an isomorphism of $\sB^\Delta_R$-modules
	\[
	 (\sB^\Delta_R dX)^{\otimes 2}\to x^{-2}\sB^\Delta_R dx.
	\]
	Hence the theorem follows from Lemma \ref{LemExt}.
\end{proof}

\subsection{Arithmetic de Rham pairing}\label{SubsecArithDRPairing}

Let $S=\Spec(B)$ be an affine scheme over $A[1/\frn]$ and let $E=(\cL_E,\Phi^E)$ be a Drinfeld module of rank two over $S$ with 
a $\Gamma_1^\Delta(\frn)$-structure $(\lambda,\mu)$.
Let $f:S\to Y_1^\Delta(\frn)$ be the morphism that the isomorphism class of $(E,\lambda,\mu)$ defines. We define a map
\[
\langle-,-\rangle_E:\DR(E,\Ga)\otimes_{B}\DR(E,\Ga)\to B
\]
as follows. By (\ref{EqnIdentifyDR}), we have an isomorphism
\[
\cL_E^{\otimes -q}(S)\oplus \cL_E^{\otimes -q^2}(S)\to \DR(E,\Ga)
\]
by which we identify the source with the target. For any element $\varphi\in \DR(E,\Ga)$, write
\[
\varphi=(\varphi_1,\varphi_2),\quad \varphi_i\in \cL_E^{\otimes -q^i}(S).
\]

Let $\varphi,\varphi'\in \DR(E,\Ga)$. Following \cite[(7.14)]{Gek_dR}, we define
\[
\langle\varphi,\varphi'\rangle_E:=(\varphi_1\otimes \varphi'_2-\varphi'_1\otimes \varphi_2)\otimes (f^*\mathfrak{h})^{\otimes -q}\in B.
\]
Since $f^*\mathfrak{h}$ is nowhere vanishing, this pairing is $B$-linear, alternating and perfect. In particular, we have an $\cO_{Y_1^\Delta(\frn)_R}$-linear
alternating perfect pairing
\[
\langle-,-\rangle_\univ^\Delta:\cH_{\dR,\univ}^\Delta\otimes_{\cO_{Y_1^\Delta(\frn)_R}}\cH_{\dR,\univ}^\Delta\to \cO_{Y_1^\Delta(\frn)_R},
\]
which we call the arithmetic de Rham pairing.

\begin{lem}\label{LemArithDeRhamPairVSHodgeFil}
	The pairing
	\[
	\omega^\Delta_\univ\otimes_{\cO_{Y_1^\Delta(\frn)_R}} (\omega^\Delta_\univ)^\vee\to \cO_{Y_1^\Delta(\frn)_R},\quad 
	\varphi\otimes\prjt
	(\varphi')\mapsto \langle\varphi,\varphi'\rangle_\univ^\Delta
	\]
	that the exact sequence (\ref{EqnHodgeFilUniv}) and the arithmetic de Rham pairing induce is equal to the canonical pairing.
\end{lem}
\begin{proof}
	It is enough to show the equality of the pairings on any affine open subscheme $U=\Spec(B)$ of $Y_1^\Delta(\frn)_R$ 
	on which $\cL_{E^\Delta_\univ}$ is trivial. Put $E=E^\Delta_\univ|_U$.
	Write $E=\Spec(B[X])$, $E^D=\Spec(B[Y])$ as (\ref{EqnIdentificationXY}) and
	\[
	\Phi^E_t=\theta+ \mathbf{a}_1\tau +\mathbf{a}_2\tau^2\quad (\mathbf{a}_1\in B,\ \mathbf{a}_2\in B^\times).
	\]
	Then we have
	\[
	\mathfrak{h}|_U=b (dX)^{\otimes q+1},\quad b\in B^\times
	\]
	so that $b^{q-1}=-\mathbf{a}_2$.
	
	By (\ref{EqnELie}), the element
	\[
	\tau\in \DR(E,\Ga)=B\tau\oplus B\tau^2
	\]
	satisfies $\prjt(b\tau)=\frac{d}{dX}\in \Lie(E)$ and
	\[
	\langle dX,b\tau\rangle_E=\langle\mathbf{a}_1\tau+\mathbf{a}_2\tau^2,b\tau \rangle_E=-\mathbf{a}_2 b^{1-q}=1.
	\]
	This means that the induced pairing $BdX\otimes_B B\frac{d}{dX}\to B$ is the canonical one.
\end{proof}

\begin{thm}\label{ThmExtDeRhamPair}
	The arithmetic de Rham pairing extends to an $\cO_{X_1^\Delta(\frn)_R}$-linear
	alternating perfect pairing
	\[
	\overline{\langle-,-\rangle}_\univ^\Delta:\bar{\cH}_{\dR,\univ}^\Delta\otimes_{\cO_{X_1^\Delta(\frn)_R}}\bar{\cH}_{\dR,\univ}^\Delta\to 
	\cO_{X_1^\Delta(\frn)_R}
	\]
	such that the pairing 
	\[
	\bar{\omega}^\Delta_\univ\otimes_{\cO_{X_1^\Delta(\frn)_R}} (\bar{\omega}^\Delta_\univ)^\vee\to\cO_{X_1^\Delta(\frn)_R},\quad\varphi\otimes\bar{\prjt}
	(\varphi')\mapsto \overline{\langle\varphi,\varphi'\rangle}_\univ^\Delta
	\]
	that the exact sequence of Theorem \ref{ThmExtProjToLie}
	induces is equal to the canonical pairing.
\end{thm}
\begin{proof}
	For the existence of an extension as an alternating $\cO_{X_1^\Delta(\frn)_R}$-linear map, consider the pull-back 
	\[
	\langle-,-\rangle_{B^\Delta_R}:=(T^\td)^*\langle-,-\rangle_\univ^\Delta.
	\]
	Let $dX$ and $\eta$ be the elements of (\ref{EqnDfnDX}) and (\ref{EqnDfnEta}), respectively. By (\ref{EqnDfnLx}) and 
	Corollary \ref{CorH-FunctionOnP}, we have
	\begin{equation}\label{EqnDeRhamPairOnB}
	\langle dX,\eta\rangle_{B^\Delta_R}=x^2l(x)a_2b_h^{-q}\in (\sB_R^\Delta)^\times.
	\end{equation}
	Thus we obtain an extension as an $\cO_{X_1^\Delta(\frn)_R}$-linear map by Lemma \ref{LemExt}. Moreover, since $X_1^\Delta(\frn)_R$ is an integral scheme,
	any nonempty open subset $U$ of it meets $Y_1^\Delta(\frn)_R$ and the restriction map
	\begin{equation}\label{EqnRestrictionMap}
		\cO(U)\to \cO(U\cap Y_1^\Delta(\frn)_R)
	\end{equation}	
	is injective. 
	Since $\langle-,-\rangle_\univ^\Delta$ is alternating, for any section $\varphi\in \bar{\cH}_{\dR,\univ}^\Delta(U)$ the element 
	$\overline{\langle \varphi,\varphi\rangle}_\univ^\Delta$ vanishes on $U\cap Y_1^\Delta(\frn)_R$ and thus $\overline{\langle \varphi,\varphi\rangle}
	_\univ^\Delta=0$.

	For the perfectness of $\overline{\langle-,-\rangle}_\univ^\Delta$, consider the map
	\[
	\bar{\cH}_{\dR,\univ}^\Delta\to \sHom_{\cO_{X_1^\Delta(\frn)_R}}(\bar{\cH}_{\dR,\univ}^\Delta,\cO_{X_1^\Delta(\frn)_R}),\quad \varphi\mapsto (\varphi'\mapsto 
	\overline{\langle\varphi,\varphi'\rangle}_\univ^\Delta).
	\]
	Then its pull-back by $\sT^\td$ is given by the $\sB^\Delta_R$-linear map
	\[
	(dX,\eta)\mapsto ((dX)^\vee, \eta^\vee)\begin{pmatrix}
		0& -x^2l(x)a_2b_h^{-q}\\
		x^2l(x)a_2b_h^{-q}& 0
	\end{pmatrix},
	\]
	where $\{(dX)^\vee,\eta^\vee\}$ is the dual basis. By (\ref{EqnDeRhamPairOnB}), it is an isomorphism of $\sB_R^\Delta$-modules. 
	Hence Lemma \ref{LemExt} shows that the map above is an isomorphism.
	Since we have proved that $\overline{\langle-,-\rangle}_\univ^\Delta$ is alternating,
	we obtain the perfectness.
	
	For the last statement, by Lemma \ref{LemArithDeRhamPairVSHodgeFil}, the two pairings are equal on $Y_1^\Delta(\frn)_R$.
	Then the equality on $X_1^\Delta(\frn)_R$ follows from the injectivity of (\ref{EqnRestrictionMap}).
\end{proof}



\end{document}